\documentclass[a4paper,11pt]{article}

\usepackage{blindtext,titlefoot}

\usepackage{amsfonts,amsthm,amssymb,amsmath}
\usepackage[title,titletoc,toc]{appendix}
\usepackage[utf8]{inputenc}
\usepackage[numbers]{natbib}

 \usepackage{amsthm,amsmath}
 \usepackage[numbers]{natbib}

\usepackage{lineno,hyperref}
\usepackage[affil-it]{authblk}
\usepackage{xcolor}
\usepackage{tikz}
\newcommand\COMP{\hbox{C\kern -.58em {\raise .54ex \hbox{$\scriptscriptstyle |$}}
\kern-.55em {\raise .53ex \hbox{$\scriptscriptstyle |$}} }}
\newcommand\NN{\hbox{I\kern-.2em\hbox{N}}}
\newcommand\RR{\hbox{I\kern-.2em\hbox{R}}}
\newcommand\sRR{{\it \hbox{I\kern-.2em\hbox{R}}}}
\newcommand\QQ{\hbox{I\kern-.53em\hbox{Q}}}
\newcommand\PP{\hbox{I\kern-.53em\hbox{P}}}
\newcommand\EE{\hbox{I\kern-.53em\hbox{E}}}
\newcommand\ZZ{{{\rm Z}\kern-.28em{\rm Z}}}
\newcommand\be{\begin{equation}}
\newcommand\ee{\end{equation}}
\newcommand{\norm}[1]{\left\lVert#1\right\rVert}

\DeclareMathOperator*{\essinf}{ess\,inf}

\DeclareMathOperator*{\esssup}{ess\,sup}

\newtheorem{theorem}{Theorem}[section]

\newtheorem{proposition}[theorem]{Proposition}
\newtheorem{remark}[theorem]{Remark}

\newtheorem{example}[theorem]{Example}
\newtheorem{lemma}[theorem]{Lemma}

\newtheorem{definition}[theorem]{Definition}

\makeatletter
\newcommand*\bigcdot{\mathpalette\bigcdot@{.5}}
\newcommand*\bigcdot@[2]{\mathbin{\vcenter{\hbox{\scalebox{#2}{$\m@th#1\bullet$}}}}}
\makeatother
\newcommand{\is}{\bigcdot }

\def \Lbrack {[\![}
\def \Rbrack {]\!]}

\numberwithin{equation}{section}

\begin{document}
\title{Linear reflected backward stochastic differential equations arising from vulnerable claims in markets with random horizon}
\author[1]{Tahir Choulli\thanks{tchoulli@ualberta.ca}}
\author[2]{Safa' Alsheyab\thanks{smalsheyab6@just.edu.jo}}
\affil[1]{Mathematical and Statistical Sciences, University of  Alberta, Edmonton, AB, Canada}
\affil[2]{Department of Mathematics and Statistics, Jordan University of Science and Technology, Jordan, }


\renewcommand\Authands{ and }





\maketitle
{\Large\bf{This paper is the second main part of the paper: Optimal stopping problem under random horizon: Mathematical structures and linear RBSDEs,  
available at https://arxiv.org/abs/2301.09836v1.}}

\abstract{ This paper considers the setting governed by $(\mathbb{F},\tau)$, where $\mathbb{F}$  is the ``public" flow of information, and $\tau$ is a random time which might not be $\mathbb{F}$-observable. This framework covers credit risk theory and life insurance. In this setting, we assume $\mathbb{F}$ being generated by a Brownian motion  $W$ and consider a vulnerable claim $\xi$, whose payment's policy depends {\it{essentially}} on the occurrence of $\tau$.  The hedging  problems, in many directions, for this claim led to the question of studying the linear reflected-backward-stochastic differential equations (RBSDE hereafter),  
\begin{equation*}
\begin{split}
&dY_t=f(t)d(t\wedge\tau)+Z_tdW_{t\wedge{\tau}}+dM_t-dK_t,\quad Y_{\tau}=\xi,\\
& Y\geq S\quad\mbox{on}\quad \Lbrack0,\tau\Lbrack,\quad \displaystyle\int_0^{\tau}(Y_{s-}-S_{s-})dK_s=0\quad P\mbox{-a.s.}.\end{split}
\end{equation*}
This is the objective of this paper. For this RBSDE and without any further assumption on $\tau$ that might neglect any risk intrinsic to its stochasticity, we answer the following: a) What are the sufficient minimal conditions on  the data $(f, \xi, S, \tau)$ that guarantee the existence of the solution to this RBSDE? b) How can we estimate the solution in norm using $(f, \xi, S)$? c) Is there an $\mathbb F$-RBSDE that is intimately related to the current one and how their solutions are related to each other? This latter question has practical and theoretical leitmotivs. 

}

\maketitle

\section{Introduction}\label{sec1}
Our framework in this paper falls into the setting of informational models, where there are two flows of information. The ``public" flow, which will be denoted by $\mathbb{F}:=({\cal{F}}_t)_{t\geq 0}$ and is available to all agents, and a larger flow $\mathbb{G}:=({\cal{G}}_t)_{t\geq0}$ containing additional information about the occurrence of arbitrary random time $\tau$. As $\tau$ cannot be observed before it occurs, then  $\mathbb{G}$ is the progressive enlargement of $\mathbb{F}$ with $\tau$, and makes $\tau$ a $\mathbb{G}$-stopping time. This setting covers many domains in finance and insurance. In fact, $\tau$ might represent the default of a firm or a client in credit risk theory, the death time of an insured in life insurance,  the job's termination time of an {\it employee-stock-option}'s holder (called ESO hereafter) in finance, ..., etcetera. Hence, our current setting covers these three aforementioned frameworks, and for more about these we refer the reader to \cite{BelangerShreveWong,BieleckiRutkowski,Szimayer,CarrLinetsky} and the references therein to cite a few. \\

By addressing the exponential hedging of vulnerable claims in \cite{Choulli5}, see \cite{Alsheyab} for a very early version, and the Esscher pricing for vulnerable claims in \cite{Choulli6}, we arrived at a crucial and important question that points to the study of the following RBSDE
\begin{align}\label{RBSDE(General)}
\begin{cases}
dY_{t}=-f(t, Y_t,Z_t)d(t\wedge\tau)-d(K_{t}+M_{t})+Z_{t}dW_{t\wedge\tau},\quad {Y}_{\tau}=\xi,\\
 Y\geq S\quad\mbox{on}\quad\Lbrack 0,\tau\Lbrack,\quad\mbox{and}\quad E\left[\displaystyle\int_{0}^{\tau}(Y_{t-}-S_{t-})dK_{t}\right]=0.
\end{cases}
\end{align}
Here $W$ is a Brownian motion, and  $\mathbb F$ is assume to be generated by $W$, for simplicity of exposition only. The triplet $(f,S,\xi)$ is the data of the model in which $\xi$ is the vulnerable claim, $Y$is the value process that we are looking to describe, and $S$ is the lower barrier for $Y$ which is a RCLL and $\mathbb F$-adapted process with values in $[-\infty, +\infty)$, and importantly the driver $f(t,x,y)$ is measurable and Lipchitz in $(y,z)$. We illustrate our claim, that (\ref{RBSDE(General)}) with nice driver $f$ is the unified form of BSDEs arising from vulnerable claims when $\tau$ is left to be arbitrary general, by a simple example.
\begin{example} Consider a simple complete market model, which consists of a non-risky asset $S_0(t)=\exp(\int_0^tr_s ds)$, where $r$ is nonnegative and bounded, and one risky asset $S_1$ following the geometric Brownian motion, such that the market price of risk $\lambda$ is a bounded process.  Let $T$ be a finite fixed investment horizon, and $g$ be a nonnegative and bounded $\mathcal F_T$-measurable random variable. Then the vulnerable claim has its payoff $B:=gI_{\{\tau>T\}}$. This is a vulnerable option without recovery, which guarantees the payment of $g$ at time $T$ if the insured survives at this time and nothing otherwise. For the exponential hedging problem of this claim, we deal with the optimization problems:
\begin{equation}\label{PrimalDualProblems}
\underbrace{\sup_{\theta\in\Theta(S^{\tau},\mathbb{G})}E\left[1-\exp(-\gamma(X^{\theta}_{T\wedge\tau}-B))\right]}_{=:1-\widetilde{\cal{P}}_0}=1-\exp\biggl(-\underbrace{\inf_{Z\in{\cal{D}}(S^{\tau},\mathbb{G})}\Biggl\{E[Z_T\ln(Z_T)]-\gamma{E}[Z_TB]\Biggr\}}_{=:1-\exp(-\widetilde{J}_0)}\biggr).
\end{equation}
Here, $\gamma>0$, $\Theta(S^{\tau},\mathbb{G})$ is the set of all ``admissible" strategies, and ${\cal{D}}(S^{\tau},\mathbb{G})$ is the dual set (the set of all deflators for $(S^{\tau},\mathbb{G})$), and $X^{\theta}$ is the wealth process associated to the strategy $\theta$ with zero initial wealth.\\
\end{example}
For more details about the equality in (\ref{PrimalDualProblems}), for the general setting of semimartingales and the various form of $\Theta(S^{\tau},\mathbb{G})$, we refer the reader to the seminal paper \cite{DelbaenFiveAuthors}  and the references therein to cite a few. For our simple ``initial model" $(S,\mathbb{F})$, while $\tau$ is left to be arbitrary general, we focus on the two valuation processes $\widetilde{\cal{P}}$ and $\widetilde{J}$ given by 
\begin{equation*}
\begin{split}
&\widetilde{\cal{P}}_t:=\essinf_{\theta\in\Theta(S^{\tau},\mathbb{G})}E\left[\exp\left(-\gamma(X^{\theta}_{T\wedge\tau}-X^{\theta}_{t\wedge\tau}-B)\right)\Big|\mathcal{G}_{t}\right],\\
&\widetilde{J}_t:=\esssup_{Z\in{\cal{D}}(S^{\tau},\mathbb{G})}E\left[-{{Z_T}\over{Z_t}}\ln\left(\frac{Z_T}{Z_t}\right)+{\gamma}{{Z_T}\over{Z_t}}B\Big|\mathcal{G}_{t}\right].
\end{split}
\end{equation*}
Indeed, we prove in \cite{Choulli5} -- see \cite{Alsheyab} for an early version of a part of this work--, that there exists a right-continuous-with-left-limits (RCLL hereafter for short) and $\mathbb{F}$-adapted process $\widehat{S}$, and a nicely integrable ${\cal{G}}_T$-measurable random variable $\widehat{B}$ such that $\widetilde{J}$ is the solution to the following RBSDE
\begin{equation*}
\begin{split}
dY_{t}=-({1\over{2}}\lambda^2_t+\lambda_t Z_t)d(t\wedge\tau)-d(K_{t}+M_{t})+Z_{t}dW_{t\wedge\tau},\quad {Y}_{\tau}=-\gamma(B-\widehat{B}),\\
 Y\geq \widehat{S}\quad\mbox{on}\quad\Lbrack 0,\tau\Lbrack,\quad\mbox{and}\quad E\left[\displaystyle\int_{0}^{\tau}(Y_{t-}-\widehat{S}_{t-})dK_{t}\right]=0.
\end{split}
\end{equation*}
By working under the risk neural measure of the initial model $(S,\mathbb{F})$ instead of $P$, Girsanov's theorem allows us to conclude that the above RBSDE is ``equivalent" to the following linear RBSDE
\begin{equation}\label{RBSDE(MainLinear)}
\begin{split}
&dY_{t}=-f(t)d(t\wedge\tau)-d(K_{t}+M_{t})+Z_{t}dW_{t\wedge\tau},\quad {Y}_{\tau}=\xi,\\
& Y\geq S\quad\mbox{on}\quad\Lbrack 0,\tau\Lbrack,\quad\mbox{and}\quad E\left[\displaystyle\int_{0}^{\tau}(Y_{t-}-S_{t-})dK_{t}\right]=0.
\end{split}
\end{equation}
Thus, this resulting linear RBSDE is directly used -- and it is the simplest one -- in addressing the exponential hedging problem for vulnerable claims, as the RBSDE for $\widetilde{\cal{P}}$ (the valuation primal process) takes the more general form (\ref{RBSDE(General)}). Besides these facts, it is well known nowadays in the literature of BSDEs that addressing (\ref{RBSDE(MainLinear)}) is the crucial step into solving the general case (\ref{RBSDE(General)}).  \\

{\it What is novel in (\ref{RBSDE(General)}) and (\ref{RBSDE(MainLinear)}) and what are the challenges?} To answer these questions and in virtue of our main concern of random horizon, we classify the literature on BSDEs into three major groups. The first main group, which represents the huge majority of the literature on BSDEs, treats the case when $\tau=T\in(0,\infty]$ is a fixed (finite or infinite) horizon. For this case with its extensive and various generalizations in many aspects, we refer the reader to \cite{Essaky2,Pardoux4,Bouchard,Klimsiak1,Klimsiak2,Pardoux,SonerRTouziZhang, ElqarouiBSDE,Hamadane0, Hamadane1,Hamadane2,Touzi,Elotmani} and the references therein to cite a few.\\
The second group extends the first group to the setting where $\tau$ is a finite $\mathbb{F}$-stopping time. This extension was mainly motivated by the viscosity solutions to semilinear elliptic PDE,  and we refer the reader to \cite{Peng91,PardouxPradeillesRao,BriandConfortola,Darling,Touzi,Popier,Ouknine}  and the references therein to cite a few. The third group, which has been essentially motivated by the hedging and pricing problems in credit risk, consists of allowing $\tau$ to be a finite random time that might not be an $\mathbb{F}$-stopping time. However, up to our knowledge, all the literature about this case impose conditions on $\tau$, and the most frequent of these assumptions is  the immersion assumption (also known as the H-hypothesis). This immersion assumption {\it{forces}} $(W_{t\wedge\tau})_{t\geq0}$ {\it{to remain a martingale}} under $\mathbb{G}$, see \cite{Grigorova,Ankirchner,Biagini,Kharroubi, SekineTanaka,Dong} and the references therein to cite a few.  The immersion, in particular, and other assumptions on $\tau$ boil down in a way or another into neglecting some of the risks induced by the stochasticity of $\tau$. Thus, herein this paper, we let $\tau$ be as general as the informational risks generated by $\tau$ are not affected. This implies that $W^{\tau}$ is no longer a local martingale, and hence the existing machinery of BSDEs fails to be applicable directly to  (\ref{RBSDE(MainLinear)}) as the Doob and Burkholder-Davis-Gunndy's inequalities fail for $\mathbb{F}$-local martingale stopped at $\tau$. This is one of the main novelties and challenges in our current setting. A natural and intuitive remedy to this challenge aforementioned, can be the substitution of $W^{\tau}- \beta\is(t\wedge\tau)$, which is a $\mathbb{G}$-local martingale, to $W^{\tau}$, and this leads us to an equivalent RBSDE with the driver $\overline{f}(t,Y_t,Z_t):=f(t,Y_t,Z_t)+\beta_t Z_t$. This approach definitely will impose condition(s) (the weakest are some sort of integrability) on $\beta$, or equivalently condition(s) on $\tau$, which  might affect the informational risks of $\tau$ somehow. Furthermore, these latter conditions on $\tau$ trigger the question whether such $\tau$ exists after all. This explains how the classical/existing approaches to overcome the challenges fail, and we need a novel way of thinking specific to this case of informational setting.  \\

{\it What are our achievements?} We assume that the survival probabilities, i.e. $P(\tau>t\big|{\cal{F}}_t)>0$ for all $t\geq 0$. This is the weakest assumption that can be tolerated, as it does not affect the various risks intrinsic to $\tau$, see \cite{ChoulliYansori2, ChoulliDavelooseVanmaeleMortality, ChoulliDavelooseVanmaele} about these risks. Then we address fully (\ref{RBSDE(MainLinear)})  by distinguishing the two cases depending whether the random horizon is bounded or not. For the case of bounded horizon, i.e. the case of $T\wedge\tau$ where $T\in(0,\infty)$ fixed investment horizon, we single out the adequate change of probability, $\widetilde{Q}_T$, corresponding to this informational setting without any further assumption on $\tau$. This probability $\widetilde{Q}_T$ appeared naturally, in \cite{Choulli1}, when addressing arbitrage theory and construction of deflators for models stopped at $\tau$. Using this change of probability, we give priori estimates for the $L^p$-solution, for any $p\in (1,\infty)$, where the constants are universal and does not depend on the horizon in contrast to the majority of the literature. We single out the $\mathbb{F}$-REBSDE counterpart to (\ref{RBSDE(MainLinear)})  for this case of $T\wedge\tau$, and describe precisely the relationship that binds the solutions of the two RBSDEs. The case of unbounded horizon, i.e. when $T$ goes to infinity, the situation is much more complex as $\widetilde{Q}_T$ might not converge to a probability. To over come this challenge, we establish some domination inequalities between the $L^p$-norms under $P$ and $\widetilde{Q}_T$ and then take the limiting case by letting $T$ to go to infinity. This leads us to the space and its norm, that is adequate for the data-triplet $(f, S, \xi)$, and which guarantee the existence and uniqueness of the solution to (\ref{RBSDE(MainLinear)}) without any assumption on $\tau$. Furthermore, the priori estimates for the solution in the resulting space are also elaborated. Furthermore, the $\mathbb F$-RBSDE counter part to (\ref{RBSDE(MainLinear)}) is also derived and the one-to-one relationship between the solutions of the two RBSDEs is elaborated as well.  This connection is highly motivated by its importance in credit risk theory, see \cite{Biagini,SekineTanaka} and the references therein, and  life insurance where the mortality  and longevity risks are real challenges for both academia and insurance industry.\\

 This paper has four sections including the current one. The second section defines general notations, the mathematical model of the random horizon $\tau$, and its preliminaries. The third and fourth sections are devoted to the linear RBSDEs depending whether we stop the RBSDE at $\tau\wedge{T}$ for some fixed planning horizon $T\in (0,+\infty)$, or we stop at $\tau$. The paper has Appendices where we recall some crucial results and/or prove our technical lemmas.

 \section{The mathematical setting and preliminaries}
Our mathematical framework consists of  the pair $(\cal{B},\tau)$, where $\cal{B}$ is a stochastic basis and $\tau$ is random time. Precisely, ${\cal{B}}:=(\Omega, {\cal{F}},\mathbb{F}:=({\cal{F}}_t)_{t\geq 0}, P)$ is a filtered probability space satisfying the usual conditions (i.e. right continuous and complete), and $\tau$ is a nonnegative ${\cal{F}}$-measurable random variable. This random time might not be an $\mathbb{F}$-stopping time, and hence our setting falls into the setting of informational markets with two flows of information. The ``public" flow, which is available to all agents in the system, is the filtration $\mathbb{F}$, while the larger flow contains additional information about the random time $\tau$.  The rest of this section has three subsections. The first subsection defines general notation that will be used throughout the paper, while the second subsection gives definition and notation about RBSDEs.  The third subsection presents the progressive enlargement modelling associated to $\tau$ and its preliminaries.
\subsection{General notation}
By  ${\mathbb H}$ we denote an arbitrary  filtration that satisfies the usual conditions of completeness and right continuity.  For any process $X$, the $\mathbb H$-optional projection and  the $\mathbb H$-predictable projection, when they exist, will be denoted by $^{o,\mathbb H}X$  and $^{p,\mathbb H}X$ respectively. The set ${\cal M}(\mathbb H, Q)$ (respectively  ${\cal M}^{p}(\mathbb H, Q)$ for $p\in (1,+\infty)$) denotes the set of all $\mathbb H$-martingales (respectively $p$-integrable martingales) under $Q$, while ${\cal A}(\mathbb H, Q)$ denotes the set of all $\mathbb H$-optional processes that are right-continuous with left-limits (RCLL for short) with integrable variation under $Q$. When $Q=P$, we simply omit the probability for the sake of simple notation.  For an $\mathbb H$-semimartingale $X$, by $L(X,\mathbb H)$ we denote the set of $\mathbb H$-predictable processes that are $X$-integrable in the semimartingale sense.  For $\varphi\in L(X,\mathbb H)$, the resulting integral of $\varphi$ with respect to $X$ is denoted by $\varphi\is{X}$. For $\mathbb H$-local martingale $M$, we denote by $L^1_{loc}(M,\mathbb H)$ the set of $\mathbb H$-predictable processes $\varphi$ that are $X$-integrable and the resulting integral $\varphi\is M$ is an $\mathbb H$-local martingale. If ${\cal C}(\mathbb H)$ is a set of processes that are adapted to $\mathbb H$, then ${\cal C}_{loc}(\mathbb H)$ is the set of processes, $X$, for which there exists a sequence of $\mathbb H$-stopping times, $(T_n)_{n\geq 1}$, that increases to infinity and $X^{T_n}$ belongs to ${\cal C}(\mathbb H)$, for each $n\geq 1$.  The $\mathbb H$-dual optional projection and the $\mathbb H$-dual predictable projection of a process $V$ with finite variation, when they exist, will be denoted by  $V^{o,\mathbb H}$  and $V^{p,\mathbb H}$ respectively. For any real-valued $\mathbb H$-semimartingale $L$, we denote by ${\cal E}(L)$ the Dol\'eans-Dade (stochastic) exponential, which is the unique solution to $dX=X_{-}dL,\ X_0=1,$  and is given by
\begin{eqnarray}\label{DDequation}
 {\cal E}_t(L)=\exp\left(L_t-L_0-{1\over{2}}\langle L^c\rangle_t\right)\prod_{0<s\leq t}(1+\Delta L_s)e^{-\Delta L_s}.\end{eqnarray}
 Throughout the paper,  $\mathcal{J}_{\sigma_1}^{\sigma_2}(\mathbb{H})$ denotes the set of all $\mathbb{H}$-stopping times with values in $\Lbrack
\sigma_1,\sigma_2\Rbrack$, for any two $\mathbb H$-stopping times $\sigma_1$ and  $\sigma_2$  such that $\sigma_1\leq\sigma_2$ $P$-a.s..
We recall the notion of class-D-processes.
\begin{definition}Let $(X,\mathbb{H})$ be a pair of a process $X$ and a filtration $\mathbb{H}$. Then $X$ is said to be of class-$(\mathbb{H},{\cal{D}})$ if $\{X_{\sigma}:\ \sigma\ \mbox{is a finite $\mathbb{H}$-stopping time}\}$ is a uniformly integrable family of random variables.  
\end{definition}
\subsection{RBSDEs: Definition, spaces and norms}
Throughout this subsection we suppose given a complete filtered probability space $\left(\Omega, {\cal F}, \mathbb H=({\cal H}_t)_{t\geq 0}, Q\right)$, where $\mathbb H\supseteq{\mathbb F}$ and $Q$ is any probability measure absolutely continuous with respect to $P$. We start by the following definition of RBSDEs. 
\begin{definition}\label{Definition-RBSDE} Let $\sigma$ be an $\mathbb H$-stopping time, and $(f^{\mathbb H},S^{\mathbb H},\xi^{\mathbb H})$ be  a triplet such that $f^{\mathbb H}$ is $\mbox{Prog}(\mathbb H)\otimes{\cal B}(\mathbb R)\otimes{\cal B}(\mathbb R)$-measurable functional, $S^{\mathbb H}$ is a RCLL and $\mathbb H$-adapted process, and $\xi^{\mathbb H}$ is an ${\cal H}_{\sigma}$-measurable random variable. Then an $(\mathbb H, Q)$-solution to the RBSDE
\begin{equation}\label{RBSDE4definition}
\begin{split}
&dY_t=-f^{\mathbb H}(t,Y_t,Z_t)I_{\{t\leq\sigma\}}dt+Z_t dW_{t\wedge\sigma}-dM_t-dK_t,\  Y_{\sigma}=\xi^{\mathbb H},\\
 &\displaystyle{Y}\geq S^{\mathbb H}\ \mbox{on}\ \Lbrack0,\sigma\Lbrack,\quad \int_0^{\sigma}(Y_{u-}-S_{u-}^{\mathbb H})dK_u=0\quad Q\mbox{-a.s.}.\end{split}
\end{equation}
is any quadruplet $(Y^{\mathbb H}, Z^{\mathbb H},M^{\mathbb H},K^{\mathbb H})$ satisfying (\ref{RBSDE4definition}) such that $K^{\mathbb H}$ is RCLL nondecreasing and $\mathbb H$-predictable,  $M^{\mathbb H}\in {\cal M}_{0,loc}(Q,\mathbb H)$, $M^{\mathbb H}=(M^{\mathbb H})^{\sigma}$, $K^{\mathbb H}=(K^{\mathbb H})^{\sigma}$, and
\begin{eqnarray}\label{Condition1}
\int_0^{\sigma}\left( (Z_t^{\mathbb H})^2+\vert{f^{\mathbb H}}(t,Y_t^{\mathbb H},Z_t^{\mathbb H})\vert\right) dt<+\infty\quad Q\mbox{-a.s.}
\end{eqnarray}
When $Q=P$ we will simply call the quadruplet an $\mathbb H$-solution, while the filtration is also omitted when there no risk of confusion. 
\end{definition}
In this paper, we are interested in solutions that are integrable somehow. To this end, we recall the following spaces and norms that will be used throughout the paper. We denote by $\mathbb{L}^{p}(Q)$ the space of $\mathcal{F}$-measurable random variables $\xi'$, such that
\begin{equation}\parallel \xi' \parallel_{\mathbb{L}^{p}(Q)}^{p}:=E^{Q}\left[|\xi' | ^{p}\right ]<\infty .\end{equation}
$\mathbb{D}_{\sigma}(Q,p)$ is the space of  RCLL  and ${\cal F}\otimes{\cal B}(\mathbb R^+)$-measurable processes, $Y$, with $Y=Y^{\sigma}$, $\sup_{0\leq {t}\leq\sigma}\vert{Y_t}\vert=\sup_{0\leq {t}<\infty}\vert{Y_t}\vert$ on $(\sigma=\infty)$, and
\begin{equation}
\Vert{Y }\Vert_{\mathbb{D}_{\sigma}(Q,p)}^{p}:=E^{Q}\left[\sup_{0\leq {t}\leq\sigma}\vert{Y_t}\vert^p\right ]<\infty.
\end{equation}
Here ${\cal B}(\mathbb R^+)$ is the Borel $\sigma$-field of $\mathbb R^+$, and $\sigma$ is a random time. $\mathbb{S}_{\sigma}(Q,p)$ is the space of $\mbox{Prog}(\mathbb H)$-measurable processes $Z$ such that $Z=ZI_{\Lbrack0,\sigma\Rbrack}$ and 
\begin{eqnarray*}
\Vert{Z}\Vert_{\mathbb{S}_{\sigma}(Q,p)}^{p}:=E^{Q}\left[\left (\int_{0}^{\sigma}\vert{ Z_t}\vert ^{2}dt\right )^{{p}/{2}}\right ]<\infty .
\end{eqnarray*}
 For any $M\in {\cal M}_{loc}(Q,\mathbb H)$, we define its $p$-norm by 
 \begin{equation}
\Vert {M} \Vert_{{\cal M}^p(Q)}^{p}:=E^{Q}\left[[M, M]_{\infty} ^{p/2}\right ]<\infty,\end{equation}
and the $p$-norm of any $K\in {\cal A}_{loc}(Q,\mathbb H)$  at any random time $\sigma$ is given by 
\begin{eqnarray*}
 \Vert{K}\Vert_{{\cal{A}}_{\sigma}(Q,p)}^p:=E^Q\left[\left(\mbox{Var}_{\sigma}(K)\right)^p\right].\end{eqnarray*}
 Herein and throughout the paper, Var$(K)$ denotes the total variation process of $K$, and ${\cal A}^p_{\sigma}(Q,\mathbb H):=\{K\in {\cal A}_{loc}(Q,\mathbb H):
   \Vert{K}\Vert_{{\cal{A}}_{\sigma}(Q,p)}<+\infty\}$.
   \begin{definition}\label{RBSDE4Lp} Let $p\in (1,+\infty)$. An $L^p(\mathbb{H},Q)$-solution for (\ref{RBSDE4definition}) is a $(\mathbb{H},Q)$-solution  $(Y, Z,M,K)$  which belongs to the set 
$$
\mathbb{D}_{\sigma}(Q,p)\otimes \mathbb{S}_{\sigma}(Q,p)\otimes{\cal M}^p(Q,\mathbb H)\otimes{\cal A}^p_{\sigma}(Q,\mathbb H).$$\end{definition}
 Throughout the rest of the paper, we will use for the sake of simplifying notations the following two norms for any quadruplet $(Y,Z, M, K)$ belonging to $\mathbb{D}_{\sigma}(Q,p)\otimes \mathbb{S}_{\sigma}(Q,p)\otimes{\cal M}^p(Q,\mathbb H)\otimes{\cal A}^p_{\sigma}(Q,\mathbb H)$ as follows
   \begin{equation}\label{Norm4(Y,Z,K,M)}
   \begin{split}
   &\norm{\Vert(Y,Z,M,K)\Vert}_{(Q,\mathbb{H},p)}:= \Vert{Y }\Vert_{\mathbb{D}_{\sigma}(Q,p)}+  \vert\Vert(M,Z)\Vert\vert_{(Q,\mathbb{H},p)}+\Vert{K}\Vert_{{\cal{A}}_{\sigma}(Q,p)}\\
   \\
   & \vert\Vert(M,Z)\Vert\vert_{(Q,\mathbb{H},p)}:=\Vert{Z}\Vert_{\mathbb{S}_{\sigma}(Q,p)}+\Vert {M} \Vert_{{\cal M}^p(Q)}.
   \end{split}
   \end{equation}
\subsection{The random horizon and the progressive enlargement of $\mathbb F$}
In addition to this initial model $\left(\Omega, {\cal F}, \mathbb F,P\right)$, we consider an arbitrary random time, $\tau$, that might not be an $\mathbb F$-stopping time. This random time is parametrized though $\mathbb F$ by the pair $(G, \widetilde{G})$, called survival probabilities or Az\'ema supermartingales, and are given by
\begin{eqnarray}\label{GGtilde}
G_t :=\ ^{o,\mathbb F}(I_{\Lbrack0,\tau\Lbrack})_t= P(\tau > t | {\cal F}_t) \ \mbox{ and } \ \widetilde{G}_t :=\ ^{o,\mathbb F}(I_{\Lbrack0,\tau\Rbrack})_t= P(\tau \ge t | {\cal F}_t).\end{eqnarray}
Furthermore, the following process $m$ given by 
\begin{equation} \label{processm}
m := G + D^{o,\mathbb F},\quad \mbox{where}\quad{D}:=I_{\Lbrack\tau,+\infty\Lbrack},
\end{equation}
is a BMO $\mathbb F$-martingale and plays important role in our analysis. The flow of information $\mathbb G$, which incorporates both $\mathbb F$ and $\tau$, is defined as follows. 
\begin{equation}\label{processD}
\mathbb G:=({\cal G}_t)_{t\geq 0},\ {\cal G}_t:={\cal G}^0_{t+}\ \mbox{where} \ {\cal G}_t^0:={\cal F}_t\vee\sigma\left(D_s,\ s\leq t\right)\ .
\end{equation}
 Throughout the paper, on $\Omega\times [0,+\infty)$, we consider the $\mathbb F$-optional  $\sigma$-field  denoted by ${\cal O}(\mathbb F)$ and  the $\mathbb F$-progressive  $\sigma$-field denoted by $\mbox{Prog}(\mathbb F)$. Thanks to  \cite[Theorem 3]{ACJ} and \cite[Theorem 2.3 and Theorem 2.11]{ChoulliDavelooseVanmaele}, we recall 
\begin{theorem}\label{Toperator} The following assertions hold.\\
{\rm{(a)}} For any $M\in{\cal M}_{loc}(\mathbb F)$,  the process
\begin{equation} \label{processMhat}
{\cal T}(M) := M^\tau \footnote{Throughout the paper, for any random time $\sigma$ and any process $X$, we denote by $X^{\sigma}$ the stopped process given by $X^{\sigma}_t:=X_{\sigma\wedge t},\ t\geq 0$.} -{\widetilde{G}}^{-1} I_{\Rbrack 0,\tau\Rbrack} \is [M,m] +  I_{\Rbrack 0,\tau\Rbrack} \is\Big(\sum \Delta M I_{\{\widetilde G=0<G_{-}\}}\Big)^{p,\mathbb F}\end{equation}
 is a $\mathbb G$-local martingale.\\
 {\rm{(b)}}  The process 
\begin{equation} \label{processNG}
N^{\mathbb G}:=D - \widetilde{G}^{-1} I_{\Rbrack 0,\tau\Rbrack} \is D^{o,\mathbb  F}
\end{equation}
is a $\mathbb G$-martingale with integrable variation. Moreover, $H\is N^{\mathbb G}$ is a $\mathbb G$-local martingale with locally integrable variation for any $H$ belonging to
\begin{equation} \label{SpaceLNG}
{\mathcal{I}}^o_{loc}(N^{\mathbb G},\mathbb G) := \Big\{K\in \mathcal{O}(\mathbb F):\quad \vert{K}\vert G{\widetilde G}^{-1} I_{\{\widetilde{G}>0\}}\is D\in{\cal A}_{loc}(\mathbb G)\Big\}.
\end{equation}
\end{theorem}
 For any $q\in [1,+\infty)$ and a $\sigma$-algebra ${\cal H}$ on $\Omega\times [0,+\infty)$, we define
\begin{equation}\label{L1(PandD)Local}
L^q\left({\cal H}, P\otimes dD\right):=\left\{ X\ {\cal H}\mbox{-measurable}:\quad {E}[\vert X_{\tau}\vert^q I_{\{\tau<+\infty\}}]<+\infty\right\}.\end{equation}
Throughout the paper, we assume the following assumption 
\begin{eqnarray}\label{Assumption4Tau}
 G>0\quad (\mbox{i.e., $G$ is a positive process) and}\quad \tau<+\infty\quad P\mbox{-a.s.}.
\end{eqnarray}
Under the positivity of $G$, this process can be decomposed multiplicatively into two processes, which play central roles in the paper, as follows. 
\begin{lemma}\label{Decomposition4G} If $G>0$, then $\widetilde{G}>0$, $G_{-}>0$, and $G=G_0{\cal{E}}(G_{-}^{-1}\is m){\widetilde{\cal{E}}}$, where
\begin{equation}\label{EpsilonTilde}
\widetilde{\cal{E}}:={\cal E}\left(-{\widetilde{G}}^{-1}\is D^{o,\mathbb{F}}\right).\end{equation}
\end{lemma}
For this lemma and related results, we refer the reader to \cite[Lemma 2.4]{Choulli1}. Below, we recall \cite[Proposition 4.3]{Choulli1} that will play important role throughout the paper.
\begin{proposition}\label{ZandQ}Suppose that $G>0$ and consider the process
 \begin{equation}\label{Ztilde}
\widetilde{Z}:=1/{\cal E}(G_{-}^{-1}\is{m}).
\end{equation}
Then the following assertions hold.\\
{\rm{(a)}} $\widetilde{Z}^{\tau}$ is a $\mathbb G$-martingale, and for any $T\in (0,+\infty)$, $\widetilde{Q}_T$ given by 
 \begin{equation}\label{Qtilde}
 \frac{d{\widetilde{Q}_T}}{dP}:=\widetilde{Z}_{T\wedge\tau}.
\end{equation}
is well defined probability measure on ${\cal G}_{\tau\wedge T}$.\\
{\rm{(b)}} For any  $M\in {\cal M}_{loc}(\mathbb F)$, we have $M^{T\wedge \tau}\in {\cal M}_{loc}(\mathbb G, \widetilde{Q})$.  In particular $W^{T\wedge\tau}$ is a Brownian motion for $(\widetilde{Q}, \mathbb G)$, for any $T\in (0,+\infty)$.
\end{proposition}
\begin{remark} In general, the $\mathbb G$-martingale $\widetilde{Z}^{\tau}$ might not be uniformly integrable, and hence in general $\widetilde{Q}$ might not be well defined for $T=\infty$. For these facts, we refer the reader to \cite[Proposition 4.3]{Choulli1}, where conditions for $\widetilde{Z}^{\tau}$  being uniformly integrable are fully singled out when $G>0$. 
\end{remark}
We recall the following lemma.
\begin{lemma} \label{G-projection}For any nonnegative or integrable process $X$, we always have 
\begin{eqnarray*}\label{converting}E\left[X_{t}|\mathcal{G}_{t}\right]I_{\{t\ <\tau\}}={E\left[X_{t}I_{\{t\ <\tau\}}|\mathcal{F}_{t}\right]}G_t^{-1}I_{\{t\ <\tau\}}.
\end{eqnarray*}
\end{lemma}
This follows from \cite[XX.75-(c)/(d)]{DMM}. 

\section{Linear RBSDEs with bounded horizon}\label{LinearboundedSection}

Throughout the rest of the paper, we suppose given a standard Brownian motion $W=(W_t)_{t\geq 0}$, and $\mathbb F:=({\cal F}_t)_{t\geq 0}$ is its {\it completed natural filtration}. In this section, our aim lies in addressing
 \begin{equation}\label{RBSDEG}
\begin{split}
&dY_t=-f(t)d(t\wedge\tau)-d(K_t+M_t)+ZdW_{t\wedge\tau},\quad Y_{T\wedge\tau}=\xi,\\
& Y_{t}\geq S_{t};\quad 0\leq t<  T\wedge\tau,\quad \displaystyle\int_{0}^{T\wedge\tau}(Y_{t-}-S_{t-})dK_{t}=0,\ P\mbox{-a.s..}
\end{split}
\end{equation}
Hereto the data-triplet $\left(f, S, \xi\right)$ is given as follows: The driver of the RBSDE $f$ is an $\mathbb F$-progressively measurable process, the barrier of the RBSDE $S$ is a RCLL $\mathbb F$-adapted process, and the terminal condition $\xi$ satisfies 
\begin{equation}\label{sh}
 \xi=h_{T\wedge\tau},\quad\rm{and}\quad \xi\geq S_{\tau\wedge T},\quad P\rm{-a.s.},
 \end{equation}
 for  an $\mathbb{F}$-optional $h$. The first condition above, which can be relaxed completely, is equivalent to say $\xi$ is ${\cal F}_{T\wedge\tau}$-measurable. Therefore, the $\mathbb G$-triplet-data $\left(f, S, \xi\right)$ is equivalently parametrized by the $\mathbb F$-triplet-data $\left(f, S, h\right)$. 
The rest of this section is divided into two subsections. The first subsection elaborates prior estimates for the solution of the RBSDE (when it exists), while the second subsection proves the existence and uniqueness of the solution to (\ref{RBSDEG}) and describes its $\mathbb F$-RBSDE counterpart.
\subsection{Priori estimates for the solution of (\ref{RBSDEG})} 
   This subsection elaborates priori estimates for the solution of (\ref{RBSDEG}) under $\widetilde{Q}_T$. Besides being important in controlling the solution using the data-triplet, they are vital for the unbounded horizon case, when $T$ goes to infinite. Thus, these estimates should have universal constants to be useful and applicable.    
   \begin{theorem}\label{EstimatesUnderQtilde} For any $p\in (1,+\infty)$, there exists $C\in(0,\infty)$, which depends on $p$ only, such that if ($Y^{\mathbb{G}},Z^{\mathbb{G}},K^{\mathbb{G}}, M^{\mathbb{G}}$) is $(\mathbb{G},\widetilde{Q})$-solution to (\ref{RBSDEG}), then 
\begin{equation*}
\norm{\norm{\left({Y }^{\mathbb{G}},{Z}^{\mathbb{G}},{M }^{\mathbb{G}},{K}^{\mathbb{G}}\right)}}_{(\widetilde{Q},{\mathbb{G}},p)}\leq{C}\Delta_{\widetilde{Q}}(\xi,f,S^+).\end{equation*}
Here the norm $\norm{\norm{...}}_{(\widetilde{Q},{\mathbb{G}},p)}$ is defined via (\ref{Norm4(Y,Z,K,M)}) for the product space\\ $\mathbb{D}_{T\wedge\tau}(\widetilde{Q},p)\times\mathbb{S}_{T\wedge\tau}(\widetilde{Q},p)\times{\cal M}^p(\widetilde{Q})\times{\cal{A}}_{T\wedge\tau}(\widetilde{Q},p)$ (i.e $\sigma=T\wedge\tau$), and 
\begin{equation}\label{Delta(xi,f,S)}
\Delta_{\widetilde{Q}}(\xi,f,S^+):=\Vert\xi\Vert_{L^p(\widetilde{Q})}+\Vert\int_{0}^{T\wedge\tau}\vert{f}(s)\vert{d}s\Vert_{L^p(\widetilde{Q})}+\Vert{S^{+}}\Vert_{\mathbb{D}_{T\wedge\tau}(\widetilde{Q},p)}.\end{equation}
\end{theorem}
Our contribution in these priori estimate is two folds. On the one hand, our estimates have {\it universal constants}, as these depend on $p$ only and do not depend on the horizon $T$ in contrast to the majority of the BSDE literature. In fact, these estimate  in \cite{Bouchard}, which is among the most recent and is the cutting edge in  this direction, depend on $T$ and cannot be applicable to our setting. For this dependence on $T$, we also refer the reader \cite{Pardoux,Klimsiak1,KrusePopier,ZhangJiang,Ouknine} and the references therein to cite a few. On the other hand, our estimates are for any $p\in(1,\infty)$ and the proof is original. Indeed the very few in the literature --where the constant is universal-- are elaborated for $p\in(1,2]$ only, see \cite{ElJamali,Klimsiak2,Briandetal} and the references therein to cite a few. Hence, our approach, which is applicable to any mathematical model  and not only the Brownian case, extends these to any $p\in(1,\infty)$ and gives a unified approach no matter what is $p\in(1,\infty)$. \\

The proof of Theorem \ref{EstimatesUnderQtilde} relies on useful intermediate results, that we state in two lemmas below.
    \begin{lemma}\label{Lemma4.11} The following assertions hold.\\
         {\rm{(a)}} For any $T\in (0,+\infty)$ and any $t\in (0,T]$, we have 
        $E^{\widetilde{Q}}\left[D^{o,\mathbb{F}}_{T\wedge\tau} -D^{o,\mathbb{F}}_{(t\wedge\tau)-}\big|{\cal G}_{t}\right]\leq\widetilde{G}_{t\wedge\tau}\leq 1$, $P$-a.s..\\
         {\rm{(b)}} For any $T\in (0,+\infty)$ and any $t\in (0,T]$, we have 
    $E\left[\int_{t\wedge\tau}^{T\wedge\tau}{\widetilde G}^{-1}_s dD^{o,\mathbb F}_s\big| {\cal G}_t\right]\leq 1$, $P$-a.s.. Furthermore, for $a\in(0,+\infty)$, the process ${\max(a,1)}{\widetilde G}^{-1}\is{ D}^{o,\mathbb F}-\widetilde {V}^{(a)}$ is nondecreasing, where  $\widetilde {V}^{(a)}$ is defined by 
      \begin{eqnarray}\label{Vepsilon}
   \widetilde {V}^{(a)}:={{a}\over{{\widetilde G}}}\is{ D}^{o,\mathbb F}+\sum \left(-{{a\Delta D^{o,\mathbb F}}\over{\widetilde G}}+1-\left(1-{{\Delta D^{o,\mathbb F}}\over{\widetilde G}}\right)^a\right).
       \end{eqnarray}
   \end{lemma}
     The proof of this lemma is relegated to Appendix \ref{Appendix4Proofs}. The following lemma connects the solution of  (\ref{RBSDEG}) --when it exists-- to Snell envelop (the value process of an optimal stopping problem).
     \begin{lemma}\label{Solution2SnellEnvelop}Suppose that the triplet $(f, S, \xi)$ satisfies 
\begin{eqnarray}\label{MainAssumption}
E^{\widetilde{Q}}\left[\vert\xi\vert+\int_0^{T\wedge\tau}\vert f(s)\vert ds+\sup_{0\leq u\leq\tau\wedge T}S_u^+\right]<\infty,
\end{eqnarray}
 and $(Y^{\mathbb G}, Z^{\mathbb G}, M^{\mathbb G}, K^{\mathbb G})$ is a $(\mathbb{G},{\widetilde{Q}})$-solution to (\ref{RBSDEG}). Then for $t\in [0,T]$,
 \begin{eqnarray}\label{RBSDE2Snell}
Y^{\mathbb G}_t=\esssup_{\theta\in \mathcal{J}_{t\wedge\tau}^{T\wedge\tau}(\mathbb{G})}E^{\widetilde{Q}}\left[\int_{t\wedge\tau}^{\theta}f(s)ds + S_{\theta}1_{\{\theta <T\wedge\tau\}}+\xi 1_{\{\theta=T\wedge \tau\}}\ \Big|\ \mathcal{G}_{t}\right].
 \end{eqnarray}
     \end{lemma}
 The lemma can be proved directly, while we prefer to refer the reader to \cite[Corollary 2.9]{Klimsiak2} for other proofs. Below we give the proof of Theorem \ref{EstimatesUnderQtilde}.
\begin{proof}[Proof of Theorem \ref{EstimatesUnderQtilde}] We start by remarking that when $\Delta_{\widetilde{Q}}(\xi,f,S)=\infty$, then obviously the theorem holds. Hence, without loss of generality, for the rest of the proof we assume that $\Delta_{\widetilde{Q}}(\xi,f,S)<\infty$,  and we consider ($Y^{\mathbb{G}},Z^{\mathbb{G}},K^{\mathbb{G}}, M^{\mathbb{G}}$) a $(\mathbb{G},\widetilde{Q})$-solution to (\ref{RBSDEG}). The rest of the proof is divided into two parts.\\
{\bf Part 1.} Hereto, we control $\Vert{Y }^{\mathbb{G}}\Vert_{\mathbb{D}_{T\wedge\tau}(\widetilde{Q},p)}$, and  we estimate $\Vert{K}^{\mathbb{G}}_{T\wedge\tau}\Vert_{L^p(\widetilde{Q})} $ using $\Vert\vert (M^{\mathbb{G}},Z^{\mathbb{G}})\Vert\vert_{(\widetilde{Q},{\mathbb{G}},p)}$ (see (\ref{Norm4(Y,Z,K,M)}) for its definition) afterwards. Thanks to Lemma \ref{Solution2SnellEnvelop}, we conclude that $Y^{\mathbb G}$ satisfies (\ref{RBSDE2Snell}), i.e. 
    \begin{eqnarray*}
  Y_{t}^{\mathbb{G}}=\esssup_{\theta\in \mathcal{J}_{t\wedge\tau}^{T\wedge\tau}(\mathbb{G})}\hspace{2mm}E^{\widetilde{Q}}\left[\int_{t\wedge\tau}^{\theta}f(s)ds + S_{\theta}1_{\{\theta\ <T\wedge\tau\}}+\xi 1_{\{\theta =T\wedge \tau\}}\ \Big|\ \mathcal{G}_{t}\right].
  \end{eqnarray*}
By taking $\theta=T\wedge\tau\in \mathcal{J}_{t\wedge\tau}^{T\wedge\tau}(\mathbb G)$ and using $-\xi^-\leq \xi\leq \xi^+$ and $-(f(s))^- \leq f(s)\leq (f(s))^+$, we get
 \begin{eqnarray}\label{Domination4YG}
 \vert Y_{t}^{\mathbb{G}}\vert\leq  {\widetilde M}_t:=E^{\widetilde{Q}}\left[\int_0^{T\wedge\tau}\vert f(s)\vert ds + \sup_{0\leq u\leq\tau\wedge T} S_u^+ +\vert\xi\vert\ \Big|\ \mathcal{G}_{t}\right].
  \end{eqnarray}
On the one hand, by applying Doob's inequality to $\widetilde M$ under $(\widetilde Q, \mathbb G)$ we get 
\begin{equation}\label{yesyes}
\Vert{Y }^{\mathbb{G}}\Vert_{\mathbb{D}_{T\wedge\tau}(\widetilde{Q},p)}\leq \Vert\widetilde{M}\Vert_{\mathbb{D}_{T\wedge\tau}(\widetilde{Q},p)}\leq{C}_{DB} \Delta_{\widetilde{Q}}(\xi,f,S^+),
\end{equation}
where $C_{DB}$ is the universal Doob's constant. On the other hand, by combining $ K_{T\wedge\tau}^{\mathbb{G}}=Y^{\mathbb{G}}_{0}-\xi  +\int_{0}^{T\wedge\tau}f(t)dt- M^{\mathbb{G}}_{T\wedge\tau}+  \int_{0}^{T\wedge\tau}Z^{\mathbb{G}}_{s}dW_{t}$,  (\ref{Domination4YG}) for $t=0$, and the Burkholder-Davis-Gunndy (BDG for short) inequalities for the $(\widetilde{Q}, \mathbb G)$-martingales $M^{\mathbb{G}}$ and $ Z^{\mathbb{G}}\is{W}^{\tau}$, we obtain
       \begin{equation}\label{Control4KG}\Vert{K}^{\mathbb{G}}_{T\wedge\tau}\Vert_{L^p(\widetilde{Q})}  \leq 2\Delta_{\widetilde{Q}}(\xi,f,S^+)+C_{BDG}\Vert\vert (M^{\mathbb{G}},Z^{\mathbb{G}})\Vert\vert_{(\widetilde{Q},{\mathbb{G}},p)}.\end{equation}
  Here $C_{BDG}$ is the universal BDG constant. \\
{\bf Part 2.} Herein, we control $ \Vert\vert (M^{\mathbb{G}},Z^{\mathbb{G}})\Vert\vert_{(\widetilde{Q},{\mathbb{G}},p)}$ using $\Vert{Y }^{\mathbb{G}}\Vert_{\mathbb{D}_{T\wedge\tau}(\widetilde{Q},p)}$. To this end, we combine It\^o applied to $(Y^{\mathbb G})^2$ with (\ref{RBSDEG}), and derive 
\begin{align*}
d(Y^{\mathbb G})^2&=2Y_{-}^{{\mathbb{G}}}dY^{\mathbb G}+d[ Y^{\mathbb G},Y^{\mathbb G}]\nonumber\\
&=-2Y_{-}^{\mathbb G}f(\cdot)d(s\wedge\tau)-2Y_{-}^{\mathbb G}dK^{\mathbb G}+2Y_{-}^{\mathbb G} Z^{\mathbb G}dW^{\tau}-2Y_{-}^{\mathbb G}dM^{\mathbb G}\nonumber\\
&+d[M^{\mathbb G},M^{\mathbb G}]+d[K^{\mathbb G},K^{\mathbb G}]+(Z^{\mathbb G})^2d(s\wedge\tau)+2d[{K}^{\mathbb G},{M}^{\mathbb G}].
\end{align*}
As $[ M^{\mathbb{G}}, K^{\mathbb{G}}]=\Delta{K}^{\mathbb{G}}\is{M}^{\mathbb{G}}$, the above equality yields 
\begin{align}
&[M^{\mathbb G},M^{\mathbb G}]_{\tau\wedge{T}}+\int_0^{\tau\wedge{T}}(Z^{\mathbb G}_s)^2ds\nonumber\\
&\leq (1+{1\over{\epsilon}})\sup_{0\leq s\leq\tau\wedge{T}}\vert Y^{\mathbb G}_s\vert^2+(\xi)^2+\left(\int_0^{\tau\wedge{T}}\vert f(s)\vert ds\right)^2+\epsilon({K}^{\mathbb G}_{\tau\wedge{T}})^2\label{Ito2}\\
&+2\sup_{0\leq s\leq T\wedge\tau}\vert (\Delta{K}^{\mathbb G}\is{M}^{\mathbb G})_s\vert+2\sup_{0\leq s\leq\tau\wedge{T}}\vert (Y_{-}^{\mathbb G} \is (Z^{\mathbb G}\is{W}^{\tau}-{M}^{\mathbb G}))_s\vert.\nonumber
\end{align}
Furthermore, thanks to $-\Delta{Y}^{\mathbb G}=\Delta{K}^{\mathbb G}+\Delta{M}^{\mathbb G}$, we remark that 
\begin{eqnarray*}
\vert \Delta K^{\mathbb G}\vert\leq 2\overline{M}_{-},\quad\mbox{where}\quad\overline{M}_s:= E^{\widetilde{Q}}[\sup_{0\leq t\leq T\wedge\tau}\vert Y^{{\mathbb{G}}}_t\vert\ \big|{\cal G}_s].\end{eqnarray*}
Therefore, by combining this inequality, Lemma \ref{Lemma4.8FromChoulliThesis} applied to the last two terms in the right-hand-side of (\ref{Ito2}) with $a=b=p$, and Doob's inequality applied to $\overline{M}$, we get 
\begin{align*}
&\sqrt{2}\Vert\sqrt{\vert{Y}_{-}^{\mathbb G} \is (Z^{\mathbb G}\is{W}^{\tau}-{M}^{\mathbb G})\vert}\Vert_{\mathbb{D}_{T\wedge\tau}(\widetilde{Q},p)}+\sqrt{2}\Vert\sqrt{\vert \Delta{K}^{\mathbb G}\is{M}^{\mathbb G}\vert}\Vert_{\mathbb{D}_{T\wedge\tau}(\widetilde{Q},p)}\nonumber\\
&\leq2\kappa\sqrt{\Vert\overline{M}\Vert_{\mathbb{D}_{T\wedge\tau}(\widetilde{Q},p)}\Vert{M}^{\mathbb G}\Vert_{{\cal{M}}^p(\widetilde{Q})}}+2\kappa\sqrt{\Vert{Y}^{\mathbb{G}}\Vert_{\mathbb{D}_{T\wedge\tau}(\widetilde{Q},p)}\Vert\vert (M^{\mathbb{G}},Z^{\mathbb{G}})\Vert\vert_{(\widetilde{Q},{\mathbb{G}},p)}}\nonumber\\
&\leq {{\kappa^2(1+\sqrt{C_{DB}})^2}\over{\epsilon}} \Vert{Y}^{\mathbb{G}}\Vert_{\mathbb{D}_{T\wedge\tau}(\widetilde{Q},p)}+\epsilon \Vert\vert (M^{\mathbb{G}},Z^{\mathbb{G}})\Vert\vert_{(\widetilde{Q},{\mathbb{G}},p)}.\end{align*} 
The last inequality is due to Young's inequality for any $\epsilon\in (0,1)$. Thus, by inserting the above inequality and (\ref{Control4KG}) in (\ref{Ito2}) and using $\Vert \sqrt{X+Y}\Vert_{L^p(\widetilde{Q})}\geq (\Vert\sqrt{X}\Vert_{L^p(\widetilde{Q})}+\Vert\sqrt{Y}\Vert_{L^p(\widetilde{Q})})/2$, which holds for any nonnegative random variables $X$ and $Y$, we obtain 
\begin{align*}
&({1\over{2}}-\epsilon-\sqrt{\epsilon}C_{BDG})\Vert\vert (M^{\mathbb{G}},Z^{\mathbb{G}})\Vert\vert_{(\widetilde{Q},{\mathbb{G}},p)}\\
&\leq\left\{{{\kappa^2(1+\sqrt{C_{DB}})^2}\over{\epsilon}} +\sqrt{1+{1\over{\epsilon}}}\right\}\Vert{Y}^{\mathbb{G}}\Vert_{\mathbb{D}_{T\wedge\tau}(\widetilde{Q},p)}+(1+2\sqrt{\epsilon})\Delta_{\widetilde{Q}}(\xi,f,S^+).
\end{align*}
Finally, by  taking $\epsilon$ satisfying $(1/2)-\epsilon-\sqrt{\epsilon}C_{BDG}>0$ and by combining the above inequality with (\ref{yesyes}) and (\ref{Control4KG}), the proof of the theorem follows.     \end{proof} 
We end this subsection, by providing priori estimates for the difference of solutions as follows. 
 \begin{theorem}\label{EstimatesUnderQtilde1} Suppose that ($Y^{\mathbb{G},i},Z^{\mathbb{G},i},K^{\mathbb{G},i}, M^{\mathbb{G},i}$)  is solution to the RBSDE (\ref{RBSDEG}) which corresponds to  $(f^{(i)}, S^{(i)}, \xi^{(i)})$, for each $i=1,2$. Then for any $p>1$, there exist positive $C_1$ and $C_2$ that depend on $p$ only such that
  \begin{eqnarray}\label{estimate1001}
  &&\hskip -1.5cm\Vert\delta{Y}^{\mathbb{G}}\Vert_{\mathbb{D}_{T\wedge\tau}(\widetilde{Q},p)}+\Vert\delta{Z}^{\mathbb{G}}\Vert_{\mathbb{S}_{T\wedge\tau}(\widetilde{Q},p)}+\Vert\delta{M}^{\mathbb{G}}\Vert_{{\cal{M}}^p(\widetilde{Q})}\nonumber\\
    &&\hskip -1.5cm\leq{C_1}\Delta_{\widetilde{Q}}(\delta\xi,\delta f,\delta S)+C_2\sqrt{\Vert\delta S\Vert_{\mathbb{D}_{T\wedge\tau}(\widetilde{Q},p)}}\sqrt{\sum_{i=1}^2\Delta_{\widetilde{Q}}(\xi^{(i)},f^{(i)},(S^{(i)})^+)}
 ,\end{eqnarray}
  where  $\Delta_{\widetilde{Q}}(\xi^{(i)},f^{(i)},(S^{(i)})^+)$ for $i=1,2$ and $\Delta_{\widetilde{Q}}(\delta\xi,\delta f,\delta S)$ are defined via (\ref{Delta(xi,f,S)}), and $ \delta Y^{\mathbb{G}}$, $\delta Z^{\mathbb{G}}$, $\delta K^{\mathbb{G}},\delta M^{\mathbb{G}},\delta f,\delta \xi,$ and $\delta S$ are given by
  \begin{eqnarray}\label{Delta4Solution}
  \begin{split}
&  \delta Y^{\mathbb{G}}:=Y^{\mathbb{G},1}-Y^{\mathbb{G},2},\quad \delta Z^{{\mathbb{G}}}:=Z^{\mathbb{G},1}-Z^{\mathbb{G},2},\quad \delta M^{{\mathbb{G}}}:=M^{\mathbb{G},1}-M^{\mathbb{G},2},\\&\delta K^{{\mathbb{G}}}:=K^{\mathbb{G},1}-K^{\mathbb{G},2},\quad   \delta f:=f^{(1)}-f^{(2)},\quad \delta\xi:=\xi^{(1)}-\xi^{(2)},\quad  \delta S:=S^{(1)}-S^{(2)}.\end{split}
\end{eqnarray}
\end{theorem}
\begin{proof} This proof is achieved in two  parts.\\
{\bf {Part 1.}} This part controls the first term in the left-hand-side of (\ref{estimate1001}). By virtue of Lemma \ref{Solution2SnellEnvelop} and 
\begin{equation}\label{InequalityEssentialSup}
\vert\esssup_{i\in{I}}X_i-\esssup_{i\in{I}}Y_i\vert\leq\esssup_{i\in{I}}\vert{X}_i-Y_i\vert,
\end{equation} which holds for any pair of families $\{(X_i,Y_i):\ i\in{I}\}$, we derive 
    \begin{align*}
  \vert\delta  Y_{t}^{\mathbb{G}}\vert&=\vert{Y}^{\mathbb{G},1}_t-Y^{\mathbb{G},2}_t\vert\\
  & \leq \esssup_{\theta\in \mathcal{J}_{t\wedge\tau}^{T\wedge\tau}(\mathbb{G})}\left\vert{E}^{\widetilde{Q}}\left[\int_{t\wedge\tau}^{\theta}  \delta f(s)ds +  \delta S_{\theta}1_{\{\theta\ <T\wedge\tau\}}+  \delta \xi 1_{\{\theta =T\wedge \tau\}}\ \Big|\ \mathcal{G}_{t}\right]\right\vert\\
   &\leq{E}^{\widetilde{Q}}\left[\int_{0}^{T\wedge\tau} \vert \delta f(s)\vert ds +  \sup_{0\leq s\leq T\wedge\tau}\vert\delta S_{s}\vert+ \vert \delta \xi \vert \Big|\ \mathcal{G}_{t}\right]=: {\widehat{M}}_t.
  \end{align*}
By applying Doob's inequality to $\widehat{M}$ under $(\widetilde Q, \mathbb G)$, we get 
\begin{equation}\label{yesyes1}
\Vert\delta Y^{{\mathbb{G}}}\Vert_{\mathbb{D}_{T\wedge\tau}(\widetilde{Q},p)}\leq C_{DB}\Delta_{\widetilde{Q}}(\delta\xi ,\delta{f},\delta S),
\end{equation}
where $C_{DB}$ is the universal Doob's constant that depends on $p$ only.\\
{\bf Part 2.} Hereto, we focus on $\Vert\vert (\delta{M}^{\mathbb{G}},\delta{Z}^{\mathbb{G}})\Vert\vert_{(\widetilde{Q},\mathbb{G},p)}$. To this end, we put
\begin{align*}
{\cal Q}^{\mathbb G}&:= [\delta{M}^{\mathbb G},\delta{M}^{\mathbb G}]+\int_0^{\cdot} (\delta{Z}^{\mathbb G}_s)^2d(s\wedge\tau),\quad\mbox{and}\\
\Gamma^{\mathbb G}&:=2\displaystyle\sup_{0\leq{t}\leq\cdot}\vert (\delta{Y}_{-}^{\mathbb G}\is(\delta{Z}^{\mathbb G}\is{W}^{\tau}-\delta{M}^{\mathbb G}))_t\vert+2\sup_{0\leq{t}\leq\cdot}\vert(\Delta(\delta{K}^{\mathbb G})\is\delta{M})_t\vert.
\end{align*}
and apply It\^o to $(\delta{Y}^{\mathbb G})^2$ to get 
  \begin{align}
{\cal Q}^{\mathbb G}&\leq (\delta{Y}^{\mathbb G})^2+2\int_{0}^{\cdot}\delta  Y_{s-}^{\mathbb{G}}\delta  f(s)d{s}+2\delta  Y_{-}^{\mathbb{G}}\is\delta{K}^{\mathbb{G}}+\Gamma^{\mathbb G},\nonumber\\
 & \leq 2\sup_{0\leq{t}\leq\cdot}(\delta{Y}^{\mathbb G}_t)^2+\left(\int_{0}^{\cdot}\vert\delta  f(s)\vert{d}{s}\right)^2+2\delta  S_{-}\is\delta{K}^{\mathbb{G}}+ \Gamma^{\mathbb G}.\label{Ito10}
 \end{align}
 The last inequality follows from Skorokhod's condition (i.e. $(\delta{Y}_{-}^{\mathbb{G}}-\delta  S_{-})\is\delta{K}$ is non-increasing) and Young's inequality. Furthermore, thanks to  (\ref{RBSDEG}), we deduce that 
 \begin{eqnarray*}
 \vert\Delta(\delta{K}^{\mathbb{G}})\vert\leq \widehat{N}_{-},\quad\mbox{where}\quad \widehat{N}_t:=2E^{\widetilde{Q}}[\sup_{0\leq{s}\leq{T}\wedge\tau}\vert\delta{Y}_{s}^{\mathbb{G}}\vert\ \big|{\cal G}_t].
 \end{eqnarray*}
 Thus, by combining this inequality and Lemma \ref{Lemma4.8FromChoulliThesis} applied to $\Gamma^{\mathbb G}$ with $a=b=p$, and using Doob's inequality for $\widehat{N}$ afterwards, we derive 
  \begin{equation}\label{GammaGControl}
  \begin{split}
 \Vert\sqrt{\Gamma^{\mathbb G}_{\tau\wedge{T}}}\Vert_{L^p(\widetilde{Q})} \leq {{\kappa^2(1+(1+\sqrt{C_{DB}})^2)}\over{\epsilon}}\Vert\delta Y^{\mathbb{G}}\Vert_{\mathbb{D}_{T\wedge\tau}(\widetilde{Q},p)}+\epsilon \Vert\vert (\delta{M}^{\mathbb{G}},\delta{Z}^{\mathbb{G}})\Vert\vert_{(\widetilde{Q},{\mathbb{G}},p)},\end{split}
  \end{equation}
  and 
   \begin{align*}
   {1\over{2}} \Vert\vert (\delta{M}^{\mathbb{G}},\delta{Z}^{\mathbb{G}})\Vert\vert_{(\widetilde{Q},{\mathbb{G}},p)}&\leq \Vert\sqrt{{\cal Q}_{T\wedge\tau}^{\mathbb G}}\Vert_{L^p(\widetilde{Q})}\\
  &\leq \sqrt{2}\Vert\delta Y^{\mathbb{G}}\Vert_{\mathbb{D}_{T\wedge\tau}(\widetilde{Q},p)}+\Delta_{\widetilde{Q}}(\delta\xi,\delta{f},\delta{S})\\
  &\quad+\sqrt{2}\Vert\delta{S}\Vert_{\mathbb{D}_{T\wedge\tau}(\widetilde{Q},p)}^{1/2}\Vert\delta{K}^{\mathbb G}\Vert_{{\cal{A}}_{T\wedge\tau}(\widetilde{Q},p)}^{1/2}+ \Vert\sqrt{\Gamma^{\mathbb G}_{\tau\wedge{T}}}\Vert_{L^p(\widetilde{Q})}.
   \end{align*}
 Thus, by combining the latter inequality with (\ref{GammaGControl}), (\ref{yesyes1}), the fact that  $\mbox{Var}(\delta{K}^{\mathbb G})\leq {K}^{\mathbb G,1}+{K}^{\mathbb G,2}$, and Theorem \ref{EstimatesUnderQtilde} applied to each ${K}^{\mathbb G, i}$, $i=1,2$, the proof of the theorem follows with $\epsilon\in(0,0.5)$ and 
 $$C_1=C_{DB}+ {{\epsilon+\epsilon\sqrt{2}C_{DB}+C_{DB}\kappa^2(1+(1+\sqrt{C_{DB}})^2)}\over{\epsilon(0.5-\epsilon)}},\ C_2:={{\sqrt{2C}}\over{0.5-\epsilon}}.$$ Here $C$ is the universal constant of Theorem \ref{EstimatesUnderQtilde}. This completes the proof of the theorem.
     \end{proof} 
\subsection{Existence for $\mathbb G$-RBSDE and relationship to $\mathbb F$-RBSDE}\label{Subsection4.1}
In this subsection, we prove the existence and the uniqueness of the solution to the RBSDE (\ref{RBSDEG}), we establish explicit connection between this RBSDE and its $\mathbb F$-RBSDE counterpart, and we highlight the explicit relationship between their solutions as well. 

\begin{theorem}\label{abcde}Suppose that (\ref{MainAssumption}) holds, and let ${\widetilde{\cal E}}$ be defined in (\ref{EpsilonTilde}). Consider the processes $(f^{\mathbb{F}},S^{\mathbb{F}},\xi^{\mathbb{F}})$ and $V^{\mathbb F}$  given by 
  \begin{eqnarray}
  (f^{\mathbb{F}}, S^{\mathbb{F}},  \xi^{\mathbb{F}}):=({\widetilde{\cal E}}f,{\widetilde{\cal E}}S,{\widetilde{\cal E}_T}h_{T}),\quad\rm{and}\quad V^{\mathbb F}:=1-{\widetilde{\cal E}}. \label{ProcessVFandXiF}
  \end{eqnarray}  
Then the following  assertions hold.\\
{\rm{(a)}} The following RBSDE under $\mathbb F$, associated to the triplet $ \left(f^{\mathbb{F}},S^{\mathbb{F}}, \xi^{\mathbb F}\right)$,
\begin{equation}\label{RBSDEF}
\begin{split}
&Y_{t}= \displaystyle\xi^{\mathbb{F}}+\int_{t}^{T}f^{\mathbb{F}}(s)ds+\int_{t}^{T}h_{s}dV^{\mathbb{F}}_{s}+K_{T}-K_{t}-\int_{t}^{T}Z_{s}dW_{s},\\
&Y_{t}\geq S_{t}^{\mathbb{F}}1_{\{t\ <T\}}+\xi^{\mathbb{F}}1_{\{t\ =T\}},\ 
 \displaystyle\int_{0}^{T}(Y_{t-}-S_{t-}^{\mathbb{F}})dK_{t}=0 ,\ P\mbox{-a.s.,}
\end{split}
\end{equation}
has a unique solution $(Y^{\mathbb F},  Z^{\mathbb F}, K^{\mathbb F})$ satisfying
 \begin{eqnarray}\label{RBSDE2SnellF}
 Y^{\mathbb F}_t=\esssup_{\sigma\in \mathcal{J}_{t}^{T}(\mathbb{F})}E\left[\int_{t}^{\sigma}f^{\mathbb F}_s ds+\int_{t}^{\sigma}h_s dV^{\mathbb F}_s + S_{\sigma}^{\mathbb F}1_{\{\sigma <T\}}+\xi^{\mathbb F} I_{\{\sigma =T\}}\ \Big|\ \mathcal{F}_{t}\right].
 \end{eqnarray}
{\rm{(b)}} The RBSDE (\ref{RBSDEG}) has a unique solution $(Y^{\mathbb{G}},Z^{\mathbb{G}},K^{\mathbb{G}},M^{\mathbb{G}})$ given by 
\begin{equation}  \label{secondrelation}
 \begin{split}  
 Y^{\mathbb{G}}&= \displaystyle\frac{Y^{\mathbb{F}}}{\widetilde{\cal E}}I_{\Lbrack0,T\wedge\tau\Lbrack}+\xi{I_{\Lbrack{T}\wedge\tau,+\infty\Lbrack}},\ 
  Z^{\mathbb{G}}=\frac{Z^{\mathbb{F}}}{{\widetilde{\cal E}}_{-}} I_{\Rbrack0,T\wedge\tau\Rbrack},\\
   K^{\mathbb{G}}&=\displaystyle\frac{1}{{\widetilde{\cal E}}_{-}}\is (K ^{\mathbb{F}})^{\tau}\ \mbox{and}\ 
      M^{\mathbb{G}}=\left(h-\frac{Y^{\mathbb{F}}}{{\widetilde{\cal E}}}\right)\is({N}^{\mathbb{G}})^T.\end{split}
       \end{equation}
{\rm{(c)}}  Let $p\in (1,\infty)$ and suppose that the triplet $(f, S, \xi)$ satisfies 
\begin{eqnarray}\label{MainAssumptionL(p)}
\norm{\int_0^{T\wedge\tau}\vert f(s)\vert ds+\vert\xi\vert+\sup_{0\leq u\leq\tau\wedge T}S_u^+}_{L^p(\widetilde{Q})}<\infty.
\end{eqnarray}
Then (\ref{RBSDEG}) has a unique  $L^p(\widetilde{Q},\mathbb G)$-solution $(Y^{\mathbb{G}},Z^{\mathbb{G}},K^{\mathbb{G}},M^{\mathbb{G}})$, (\ref{RBSDEF}) has a unique $L^p(P,\mathbb F)$-solution  $(Y^{\mathbb F},  Z^{\mathbb F}, K^{\mathbb F})$, and (\ref{secondrelation}) holds.
        \end{theorem}
 The proof of the theorem relies essentially on the following lemma. 
   \begin{lemma}\label{ExpecationQtilde2Pbis} Let $X$ be an $\mathbb F$-optional process, and $T\in (0,\infty)$. Then the following assertions hold.\\
     {\rm{(a)}} If $X$ is non-negative, then 
       \begin{equation}\label{XunderQtilde}
E^{\widetilde Q}[X_{T\wedge\tau}]= E\left[G_0\int_0^T X_sdV_s^{\mathbb F}+G_0X_T{\widetilde {\cal E}}_T+X_0(1-G_0)\right].
     \end{equation}
     {\rm{(b)}} If $X$ is RCLL and nondecreasing with $X_0\geq 0$, then 
     \begin{eqnarray}\label{XunderQtilde2}
   E^{\widetilde Q}[X_{T\wedge\tau}]= E\left[X_0+G_0\int_0^T{\widetilde {\cal E}}_{s-}dX_s\right].\end{eqnarray}
     {\rm{(c)}} If $X$ is RCLL and nondecreasing with $X_0=0$, then 
    \begin{eqnarray}\label{XunderQtilde3}
    \Vert({\widetilde {\cal E}}_{-}\is{X})_{T}\Vert_{L^r(P)}\leq{2}G_{0}^{-1/r}\Vert{X}_{T\wedge\tau}\Vert_{L^r(\widetilde{Q})},\quad\mbox{for any}\ r\in[1,\infty).\end{eqnarray}
     \end{lemma}
 The proof of the lemma is relegated to Appendix \ref{Appendix4Proofs}, while below we prove the previous theorem.
\begin{proof}[Proof of Theorem \ref{abcde}] We put $F_t:=\int_0^t\vert{f}_s\vert ds$ and $F^{\mathbb{F}}_t:=\int_0^t \vert{f}^{\mathbb{F}}_s\vert{ds}$, and the rest of the proof is divided into three parts.\\
{\bf Part 1.} Herein, for $r\in [1,+\infty)$, we prove that there exists $C_r\in(0,\infty)$ satisfying
\begin{equation}\label{Condition2RBSDE(F)}
 \begin{split}\norm{{F}_T^{\mathbb{F}}+\vert\xi^{\mathbb{F}}\vert+\int_0^T\vert{h}_s\vert{d}V^{\mathbb{F}}_s+\sup_{0\leq u\leq{T}}(S_u^{\mathbb{F}})^+}_{L^r(P)} \leq{C_r}G_0^{-1/r}\Vert{F}_{T\wedge\tau}+\vert\xi\vert+\sup_{0\leq u\leq\tau\wedge T}S_u^+\Vert_{L^r(\widetilde{Q})}.\end{split}\end{equation}
On the one hand, by combining (\ref{XunderQtilde}) applied to $\vert{h}\vert^r$  (see Lemma \ref{ExpecationQtilde2Pbis}) and using ${\widetilde{\cal E}}^r\leq {\widetilde{\cal E}}$ and $(\vert{h}\vert\is{V}^{\mathbb{F}})^r\leq \vert{h}\vert^r\is{V}^{\mathbb{F}}$ afterwards, we derive
\begin{align}\label{Inequality4Xi}
\norm{\vert\xi^{\mathbb{F}}\vert+\int_0^T\vert{h}_s\vert{d}V^{\mathbb{F}}_s}_{L^r(P)}&\leq \norm{{h}_T\sqrt[r]{{\widetilde {\cal E}}_T}}_{L^r(P)}+\sqrt[r]{\norm{(\vert{h}\vert^r\is{V}^{\mathbb{F}})_T}_{L^1(P)}}\nonumber\\
&\leq{2}G_0^{-1/r}\Vert\xi\Vert_{L^r(\widetilde{Q})}.\end{align}
On the other hand, we use the following facts, which hold for any RCLL and nonnegative process $V$,
\begin{equation}\label{Vprocess}
 \sup_{0\leq s\leq{t}}{\widetilde {\cal E}}_sV_s \leq\sup_{0\leq s\leq{t}}{\widetilde {\cal E}}_sV^*_s\leq{V}_0+\int_0^t {\widetilde {\cal E}}_{s-}dV^*_s,\ \mbox{where}\ V^*_t:=\sup_{0\leq{s}\leq{t}}V_s,
 \end{equation}
 and apply (\ref{XunderQtilde3}) to $F+\sup_{0\leq{u}\leq\cdot}S^+_u-S^+_0$ afterwards, we get 
\begin{align*}
\Vert{F}_T^{\mathbb{F}}+\sup_{0\leq u\leq{T}}(S_u^{\mathbb{F}})^+\Vert_{L^r(P)}&\leq \Vert{S}_0^+\Vert_{L^r(P)}+\Vert{F}_T^{\mathbb{F}}+\sup_{0\leq u\leq{T}}(S_u^{\mathbb{F}})^+-S_0^+\Vert_{L^r(P)}\\
&\leq  \Vert{S}_0^+\Vert_{L^r(\widetilde{Q})} +{{2}\over{\sqrt[r]{G_0}}}\Vert {F}_{T\wedge\tau}+\sup_{0\leq u\leq{T\wedge\tau}}S_u^+-S_0^+\Vert_{L^r(\widetilde{Q})}\\
&\leq(1+2G_0^{-1/r})\Vert {F}_{T\wedge\tau}+\sup_{0\leq u\leq{T\wedge\tau}}S_u^+\Vert_{L^r(\widetilde{Q})} .\end{align*}
Thus, by combining this with (\ref{Inequality4Xi}), (\ref{Condition2RBSDE(F)}) follows, and this ends part 1.\\
{\bf Part 2.}  This part proves assertions (a) and (b).  Thanks to part 1, for the case of $r=1$, we deduce that (\ref{MainAssumption}) implies that 
$$
E\left[\vert\xi^{\mathbb{F}}\vert+\int_0^T\vert{f}^{\mathbb{F}}_s\vert{ds}+\int_0^T\vert{h}_s\vert{d}V^{\mathbb{F}}_s +\sup_{0\leq{t}\leq{T}}(S^{\mathbb{F}}_t)^+\right]<\infty.
$$
By combining this with $S^{\mathbb{F}}_t\leq{E}\left[\sup_{0\leq{t}\leq{T}}(S^{\mathbb{F}}_t)^+\ \big|{\cal{F}}_t\right]$, we can apply directly \cite[Theorem 2.13 and Corllary 2.9]{Klimsiak2} and deduce the existence and uniqueness of the solution to (\ref{RBSDEF}) satisfying (\ref{RBSDE2SnellF}). This proves assertion (a). \\
Similarly, thanks to (\ref{MainAssumption})  and $S_t\leq{E}^{\widetilde{Q}}\left[\sup_{0\leq{t}\leq{T}}S_t^+\ \big|{\cal{G}}_t\right]$, we directly apply again \cite[Theorem 2.13]{Klimsiak2} to (\ref{RBSDEG}) under $\mathbb{G}$ and $\widetilde{Q}$, and conclude the existence and the uniqueness of the solution to this RBSDE.  Thus, the rest of this part proves that  this solution, that we denote by $(Y^{\mathbb G},Z^{\mathbb G},K^{\mathbb G}, M^{\mathbb G})$, fulfills  (\ref{secondrelation}). To this end, on the one hand, thanks to the Doob-Meyer decomposition under $(\widetilde{Q},\mathbb G)$, we remark that  for any solution $(Y,Z,K, M)$ to (\ref{RBSDEG}), we have $(Y,Z,K, M)=(Y^{\mathbb G},Z^{\mathbb G},K^{\mathbb G}, M^{\mathbb G})$ if and only if $Y=Y^{\mathbb G}$. On the other hand, due to (\ref{RBSDE2Snell}) --see Lemma \ref{Solution2SnellEnvelop}--, we have 
\begin{equation}\label{YG2StildeG}
Y^{\mathbb G}_t+\int_0^{\tau\wedge{T}\wedge{t}}f(s)ds={\cal S}^{\mathbb{G},\widetilde{Q}}_t:=\esssup_{\theta\in {\cal T}_{t\wedge\tau}^{T\wedge\tau}(\mathbb G)}E^{\widetilde{Q}}\left[X^{\mathbb{G}}_{\theta}\Big|{\cal{G}}_t\right],\end{equation}
where $  X^{\mathbb G}$ is RCLL given by 
\begin{equation*}
 X^{\mathbb G}:=\int_0^{\tau\wedge{T}\wedge\cdot}f(s)ds+SI_{\Lbrack0,\tau\wedge{T}\Lbrack}+h_{\tau\wedge{T}}I_{\Lbrack\tau\wedge{T},+\infty\Lbrack}.\end{equation*}
Therefore, in order to apply  \cite[Theorem 3-(b)]{ChoulliAndSafa}, we need to find the unique pair $(X^{\mathbb F}, k^{(pr)})$ associated to $X^{\mathbb G}$. To this end, we remark that 
\begin{eqnarray*}
SI_{\Lbrack0,\tau\wedge{T}\Lbrack}=SI_{\Lbrack0,\tau\Lbrack}I_{\Lbrack0,{T}\Lbrack}\quad\mbox{and}\quad h_{\tau\wedge{T}}I_{\Lbrack\tau\wedge{T},+\infty\Lbrack}I_{\Lbrack0,\tau\Lbrack}=h_{T}I_{\Lbrack0,\tau\Lbrack}I_{\Lbrack{T},+\infty\Lbrack},\end{eqnarray*}
and derive
\begin{eqnarray*}
X^{\mathbb F}=\int_0^{T\wedge\cdot}f(s)ds+SI_{\Lbrack0,{T}\Lbrack}+h_{T}I_{\Lbrack{T},+\infty\Lbrack},\ 
k^{(pr)}=\int_0^{T\wedge\cdot}f(s)ds+h_{T\wedge\cdot}=k^{(op)}.
\end{eqnarray*}
Therefore, the process ${\widetilde{X}}^{\mathbb{F}}$ of Theorem  \cite[Theorem 3-(b)]{ChoulliAndSafa} is given by
\begin{equation*}
{\widetilde{X}}^{\mathbb{F}}:=({\widetilde{\cal E}}X^{\mathbb F}-k^{(op)}\is{\widetilde{\cal E}})^T=\int_0^{T\wedge\cdot}f^{\mathbb F}(s)ds+(h\is V^{\mathbb F})^T+S^{\mathbb F} I_{\Lbrack0,T\Lbrack}+\xi^{\mathbb F}I_{\Lbrack{T},+\infty\Lbrack},\end{equation*}
while due to (\ref{RBSDE2SnellF}) its Snell envelop ${\widetilde{S}}^{\mathbb{F}}$ satisfies 
\begin{equation}\label{TildeF4YF}
Y^{\mathbb F}+L^{\mathbb F}={\widetilde{S}}^{\mathbb{F}},\quad L^{\mathbb F}:=\int_0^{T\wedge\cdot}f^{\mathbb F}(s)ds+\int_0^{T\wedge\cdot}h_sdV^{\mathbb F}_s.
\end{equation}
Thus, by applying \cite[Theorem 3-(b)]{ChoulliAndSafa} and using (\ref{YG2StildeG}), (\ref{TildeF4YF}), and 
\begin{eqnarray*}
k^{(op)}{\widetilde{\cal E}}-k^{(op)}\is{\widetilde{\cal E}}=L^{\mathbb F}+\widetilde{\cal E}hI_{\Lbrack0,T\Lbrack} +\xi^{\mathbb F}I_{\Lbrack{T},+\infty\Lbrack},
\end{eqnarray*}
 we obtain
\begin{align*}
&Y^{\mathbb G}+\int_0^{\tau\wedge{T}\wedge\cdot}f(s)ds={\cal S}^{\mathbb {G}, \widetilde{Q}}={{{\widetilde{\cal S}}^{\mathbb F}}\over{ \widetilde{\cal E}^T}}(I_{\Lbrack0,\tau\Lbrack})^T+{{L^{\mathbb F}+\widetilde{\cal E}hI_{\Lbrack0,T\Lbrack} +\xi^{\mathbb F}I_{\Lbrack{T},+\infty\Lbrack}}\over{\widetilde{\cal E}}}\is(N^{\mathbb G})^T\\
&={{Y^{\mathbb F}+L^{\mathbb F}}\over{ \widetilde{\cal E}^T}}(I_{\Lbrack0,\tau\Lbrack})^T+{{L^{\mathbb F}}\over{ \widetilde{\cal E}}}\is(N^{\mathbb G})^T+\left(hI_{\Lbrack0,T\Lbrack} +h_T{I}_{\Lbrack{T},+\infty\Lbrack}\right)\is(N^{\mathbb G})^T\\
&={{Y^{\mathbb F}}\over{ \widetilde{\cal E}^T}}(I_{\Lbrack0,\tau\Lbrack})^T+{1\over{ \widetilde{\cal E}_{-}}}\is (L^{\mathbb F})^{T\wedge\tau}+h\is D^T- {{h}\over{\widetilde{G}}}I_{\Rbrack0,\tau\wedge{T}\Rbrack}\is D^{o,\mathbb F}\\
&={{Y^{\mathbb F}}\over{ \widetilde{\cal E}}}I_{\Lbrack0,T\wedge\tau\Lbrack}+\int_0^{\tau\wedge{T}\wedge\cdot}f(s)ds+\xi{I}_{\Lbrack{T}\wedge\tau,+\infty\Lbrack}.
\end{align*}
The fourth equality is due to  \cite[Lemma 4-(a)]{ChoulliAndSafa} , while the last equality is due to  $h\is D^T=\xi{I}_{\{\tau\leq{T}\}}{I}_{\Lbrack{T}\wedge\tau,+\infty\Lbrack}$ and 
$${{Y^{\mathbb F}}\over{ \widetilde{\cal E}^T}}(I_{\Lbrack0,\tau\Lbrack})^T={{Y^{\mathbb F}}\over{ \widetilde{\cal E}^T}}I_{\{\tau>T\}}+{{Y^{\mathbb F}}\over{ \widetilde{\cal E}}}I_{\{\tau\leq T\}}I_{\Lbrack0,T\wedge\tau\Lbrack}={{Y^{\mathbb F}}\over{ \widetilde{\cal E}}}I_{\Lbrack0,T\wedge\tau\Lbrack}+\xi{I}_{\{\tau>T\}}{I}_{\Lbrack{T}\wedge\tau,+\infty\Lbrack}.$$
This proves assertion (b), and ends part 2.\\
{\bf Part 3.} This part proves assertion (c). To this end, we suppose that (\ref{MainAssumptionL(p)}) holds. Thanks to part 1 with $r=p$, then we deduce that 
$$ \Delta(\xi^{\mathbb{F}},f^{\mathbb{F}},(S^{\mathbb{F}})^+):=\norm{{F}_T^{\mathbb{F}}+\vert\xi^{\mathbb{F}}\vert+\int_0^T\vert{h}_s\vert{d}V^{\mathbb{F}}_s+\sup_{0\leq u\leq{T}}(S_u^{\mathbb{F}})^+}_{L^p(P)}<\infty.$$ On the one hand, by using assertion (a) and adapting literally the method in the proof of Theorem \ref{EstimatesUnderQtilde}, we deduce the existence of a constant $C>0$ depending on $p$ only such that 
\begin{equation*}
\Vert{Y }^{\mathbb{F}}\Vert_{\mathbb{D}_{T}(\mathbb{F},p)}+\Vert{Z }^{\mathbb{F}}\Vert_{\mathbb{S}_{T}(\mathbb{F},p)}
+\Vert{K}^{{\mathbb{F}}}\Vert_{{\cal{A}}_{T}(\mathbb{F},p)}\leq{C} \Delta(\xi^{\mathbb{F}},f^{\mathbb{F}},(S^{\mathbb{F}})^+).\end{equation*}
 This proves that the unique solution to (\ref{RBSDEF}) is an $L^p(P,\mathbb{F})$-solution. On the other hand, thanks to Theorems \ref{EstimatesUnderQtilde} and \ref{EstimatesUnderQtilde1} and assertion (b), we deduce that (\ref{RBSDEG}) has a unique $L^p(\widetilde{Q},\mathbb G)$-solution, which satisfies (\ref{secondrelation}). This proves assertion (c), and completes the proof of the theorem. \end{proof}
We end this section by the following remarks which will be useful.
\begin{remark}\label{Extension4Section4toSigma}
{\rm{(a)}} It is worth mentioning (it is easy to check) that the main  results of this section (especially Theorems \ref{EstimatesUnderQtilde},\ref{EstimatesUnderQtilde1} and \ref{abcde}) remain valid if we replace $T$ with any bounded $\mathbb{F}$-stopping time $\sigma$. In this case, one should use the probability $\widetilde{Q}_{\sigma}:=\widetilde{Z}_{\sigma\wedge\tau}\cdot P$ instead of $\widetilde{Q}$. \\
{\rm{(b)}} If $(Y^{\mathbb{G}},Z^{\mathbb{G}},K^{\mathbb{G}},M^{\mathbb{G}})$ is the solution of an RBSDE associated to $(f, \xi, S)$ on the interval $\Lbrack0,\sigma\wedge\tau\Rbrack$, then it is also the solution to the RBSDE associated with  $(fI_{\Lbrack0,\sigma\Rbrack}, \xi, S^{\sigma})$  on any interval $\Lbrack0,\theta\wedge\tau\Rbrack$ with $\theta\geq \sigma$.
\end{remark}
\section{Linear RBSDEs with unbounded horizon} \label{LinearUnboundedSection}
This section focuses on the following RBSDE 
\begin{equation}\label{RBSDEGinfinite}
\begin{split}
&dY=-f(t)d(t\wedge\tau)-d(K+M)+ZdW^{\tau},\quad Y_{\tau}=\xi=h_{\tau},\\
&Y_{t}\geq S_{t};\quad 0\leq t<  \tau,\quad \displaystyle\int_{0}^{\tau}(Y_{t-}-S_{t-})dK_{t}=0,\quad P\mbox{-a.s..}
\end{split}
\end{equation}
It is important to mention that $\widetilde{Q}$ (defined in (\ref{Qtilde})) depends heavily on the finite horizon planning $T$, and in general the process ${\widetilde Z}^{\tau}$ defined in (\ref{Ztilde}) might not be a $\mathbb{G}$-uniformly integrable martingale, see \cite{Choulli1} for details. Thus, the fact of letting $T$ goes to infinity triggers serious challenges in both technical and conceptual sides. In fact, both the condition (\ref{MainAssumption}) and the RBSDE (\ref{RBSDEF}) might not make sense when we take $T$ to infinity, as the limit of $h_T$ when $T$ goes to infinity might not even exist. The rest of this section is divided into two subsections. The first subsection focuses on existence and uniqueness of the solution to (\ref{RBSDEGinfinite}), while the second subsection deals with the $\mathbb F$-RBSDE counter part to it. 
 \subsection{Existence, uniqueness and priori estimates}\label{Subsection5.1} Our approach to the aforementioned  challenges has two steps. The first step relies on the following lemma, and the two theorems that follow it, where we get rid-off of $\widetilde Q$ in the left-hand-sides of our priori estimates in Theorems \ref{EstimatesUnderQtilde} and \ref{EstimatesUnderQtilde1}. The second step addresses the limits of the  terms on the right-hand-side of the estimates of these theorems.
     \begin{lemma}\label{technicallemma1} Let $T\in (0,+\infty)$, $\widetilde{Q}$ be the probability given in (\ref{Qtilde}), and $\widetilde{\cal E}$ be the process defined in (\ref{EpsilonTilde}). Then the following assertions hold.\\
     {\rm{(a)}} For any $p\in(1,+\infty)$ and any RCLL $\mathbb G$-semimartingale $Y$, we have
    \begin{eqnarray}\label{Equality4YG}
     E\left[\sup_{0\leq s\leq{T\wedge\tau}}{\widetilde{\cal E}}_s\vert{Y_s}\vert^p\right]\leq {G_0^{-1} }E^{\widetilde{Q}}\left[\sup_{0\leq s\leq{T\wedge\tau}}\vert{Y_s}\vert^p\right].     \end{eqnarray}
      {\rm{(b)}} For any $a\in (0,+\infty)$, we put $\kappa(a):=3^{1/a}(5+(\max(a, a^{-1}))^{1/a})$. Then for any RCLL, nondecreasing and $\mathbb G$-adapted $K$ with $K_0=0$, we have 
    \begin{equation}\label{Equality4KG}
     E\left[\sqrt[a]{\int_0^{T\wedge\tau}{\widetilde{\cal E}}_{s-}^{a}dK_s}\right]\leq {{\kappa(a)}\over{ G_0}} E^{\widetilde{Q}}\left[\sqrt[a]{K_{T\wedge\tau}}+\sum_{0< s\leq _{T\wedge\tau}} \widetilde{G}_s\sqrt[a]{\Delta K_s}\right].\end{equation}
        {\rm{(c)}} For any $p>1$ and any nonnegative and $\mathbb G$-optional process $H$, we have 
     \begin{equation}\label{Equality4MG}
     \begin{split}
    E\left[({\widetilde{\cal E}}_{-}^{2/p}H\is [N^{\mathbb G},N^{\mathbb G}])_{T\wedge\tau} ^{p/2}\right]\leq \kappa(2/p)G_0^{-1}  E^{\widetilde{Q}}\left[(H\is [N^{\mathbb G},N^{\mathbb G}]_{T\wedge\tau})^{p/2}+ (H^{p/2}\widetilde{G}\is \mbox{Var}(N^{\mathbb G}))_{T\wedge\tau}\right].
    \end{split}
     \end{equation}
     {\rm{(d)}} For any $p>1$ and any nonnegative and $\mathbb F$-optional process $H$, we have 
     \begin{equation}\label{Equality4MGOptionalF}
     \begin{split}
   E\left[({\widetilde{\cal E}}_{-}^{2/p}H\is [N^{\mathbb G},N^{\mathbb G}])_{T\wedge\tau} ^{p/2}\right] \leq\kappa(2/p)G_0^{-1}  E^{\widetilde{Q}}\left[(H\is [N^{\mathbb G},N^{\mathbb G}]_{T\wedge\tau})^{p/2}+ 2(H^{p/2}I_{\Rbrack0,\tau\Lbrack}\is D^{o,\mathbb F})_T\right].
    \end{split}
     \end{equation}
        \end{lemma}
To simplify the exposition, we relegate the proof of the lemma to Appendix \ref{Appendix4Proofs}. Below, we give priori estimates for the solution of  (\ref{RBSDEG}) under $P$.  
  \begin{theorem}\label{estimates}    Let $p>1$, and consider $\widetilde{\cal E}$ and $\Delta_{\widetilde{Q}}(\xi,f,S^+)$ given by (\ref{EpsilonTilde}) and (\ref{Delta(xi,f,S)}) respectively. Then there exists $\widetilde{C}\in(0,\infty)$, which depends on $p$ only, such that the $(\mathbb{G},\widetilde{Q})$-solution ($Y^{\mathbb{G}},Z^{\mathbb{G}},K^{\mathbb{G}}, M^{\mathbb{G}}$) to (\ref{RBSDEG}) satisfies  
\begin{equation*}
\norm{\norm{\left(\sqrt[p]{\widetilde{\cal E}}{Y}^{\mathbb{G}},\sqrt[p]{\widetilde{\cal E}_{-}}Z^{\mathbb{G}},\sqrt[p]{\widetilde{\cal E}_{-}}\is{M}^{\mathbb{G}} ,\sqrt[p]{\widetilde{\cal E}_{-}}\is{ K}^{\mathbb{G}}\right)}}_{(P,\mathbb{G},p)}\leq\widetilde{C} \Delta_{\widetilde{Q}}(\xi,f,S^+).
\end{equation*}
\end{theorem}

\begin{proof}  
An application of Lemma \ref{technicallemma1}-(b), to $(dK,a):=((Z^{\mathbb G})^{2}ds,{{2/p}})$, yields
 \begin{equation}\label{Estimate4ZGepsilon}
     E\left[\left(\int_0^{T\wedge\tau} ({\widetilde{\cal E}}_{s-})^{2/p}(Z^{\mathbb G}_{s})^{2}ds\right)^{{{p}\over{2}}}\right]\leq{{\kappa({{2}\over{p}})}\over{G_0}} E^{\widetilde{Q}}\left[\left(\int_0^{T\wedge\tau}
     (Z^{\mathbb G}_{s})^{2}ds\right)^{{{p}\over{2}}}\right].\end{equation}
By applying Lemma \ref{technicallemma1}-(a) to the process $Y=Y^{\mathbb G}$, we obtain
  \begin{eqnarray}\label{Estimate4YGepsilon}
     E\left[\sup_{0\leq s\leq{T\wedge\tau}}{\widetilde{\cal E}}_s\vert{Y_s^{\mathbb G}}\vert^p\right]\leq {G_0^{-1} }E^{\widetilde{Q}}\left[\sup_{0\leq s\leq{T\wedge\tau}}\vert{Y_s^{\mathbb G}}\vert^p\right]. \end{eqnarray}
Then we apply Lemma \ref{technicallemma1}-(b) to the pair $(K, a):=(K^{\mathbb G},1/p)$,  we use afterwards the fact that $\sum_{0< s\leq _{T\wedge\tau}} \widetilde{G}_s(\Delta K_s^{\mathbb G})^p\leq (K_{T\wedge\tau}^{\mathbb G})^p$, and derive
 \begin{equation}\label{Estimate4KGepsilon}
 \begin{split}
     E\left[\left(\int_0^{T\wedge\tau} ({\widetilde{\cal E}}_{s-})^{1/p}dK_s^{\mathbb G}\right)^p\right]&\leq{{\kappa(1/p)}\over{ G_0}} E^{\widetilde{Q}}\left[(K_{T\wedge\tau}^{\mathbb G})^p+\sum_{0\leq s\leq _{T\wedge\tau}} \widetilde{G}_s(\Delta K_s^{\mathbb G})^p\right]\\
     &\leq {{2\kappa(1/p)}\over{ G_0}}  E^{\widetilde{Q}}\left[(K_{T\wedge\tau}^{\mathbb G})^p\right].\end{split}\end{equation}
Thanks to Theorem \ref{abcde}-(b) --see (\ref{secondrelation})--, we have $[M^{\mathbb G}, M^{\mathbb G}]=H\is [N^{\mathbb G}, N^{\mathbb G}]$ with $H:=(h-Y^{\mathbb F}{\widetilde {\cal E}}^{-1})^2$ being a nonnegative and $\mathbb F$-optional process. Thus, a direct application of Lemma  \ref{technicallemma1}-(d) yields \\
 \begin{equation}\label{Equality4MG00}
 \begin{split}
 E\left[({\widetilde{\cal E}}_{-}^{2/p}H\is [N^{\mathbb G},N^{\mathbb G}])_{T\wedge\tau} ^{p/2}\right]\leq {{\kappa(2/p)}\over{G_0}}  E^{\widetilde{Q}}\left[(H\is [N^{\mathbb G},N^{\mathbb G}]_{T\wedge\tau})^{p/2}+ 2(H^{p/2}I_{\Rbrack0,\tau\Lbrack}\is{D}^{o,\mathbb F})_T\right].\end{split}
     \end{equation}
Furthermore, thanks to $(H^{p/2}I_{\Rbrack0,\tau\Lbrack}\is{D}^{o,\mathbb F})\leq 2^{p-1}(\vert{h}\vert^p+\vert{Y}^{\mathbb G}\vert^p)I_{\Rbrack0,\tau\Lbrack}\is{D}^{o,\mathbb F}$ and  Lemma \ref{Lemma4.11}-(b), we derive 
\begin{eqnarray*}
2E^{\widetilde{Q}}\left[(H^{p/2}I_{\Rbrack0,\tau\Lbrack}\is{D}^{o,\mathbb F})_T\right]\leq 2^pE^{\widetilde{Q}}\left[\vert{h}_{\tau}\vert^p I_{\{\tau\leq{T}\}}\right]+2^pE^{\widetilde{Q}}\left[\sup_{0\leq{t}\leq\tau\wedge{T}}\vert{Y}^{\mathbb G}_t\vert^p\right].
\end{eqnarray*}
 Therefore, by combining this inequality with ${h}_{\tau}I_{\{\tau\leq{T}\}}=\xi I_{\{\tau\leq{T}\}}$, (\ref{Equality4MG00}), (\ref{Estimate4KGepsilon}), (\ref{Estimate4YGepsilon}), (\ref{Estimate4ZGepsilon}) and Theorem \ref{EstimatesUnderQtilde}, the proof of the theorem follows immediately.
\end{proof}

Similarly, the following theorem gives a version of Theorem \ref{EstimatesUnderQtilde1} where the left-hand-side of its estimate does not involve the probability $\widetilde Q$. 

      \begin{theorem}\label{estimates1} 
       Let ($Y^{\mathbb{G},i},Z^{\mathbb{G},i},K^{\mathbb{G},i}, M^{\mathbb{G},i}$)  be a  solution to the RBSDE (\ref{RBSDEG}) corresponding to  $(f^{(i)}, S^{(i)}, \xi^{(i)})$,  for each $i=1,2$. Then there exist $\widetilde{C}_1$ and $\widetilde{C}_2$ that depend on $p$ only such that
  \begin{align*}
 & \norm{\norm{\left(\sqrt[p]{\widetilde{\cal E}}\delta{Y}^{\mathbb{G}},\sqrt[p]{\widetilde{\cal E}_{-}}\delta{Z}^{\mathbb{G}},\sqrt[p]{\widetilde{\cal E}_{-}}\is\delta{M}^{\mathbb{G}} ,0\right)}}_{(P,\mathbb{G},p)}\\
  &\leq\widetilde{C}_1\Delta_{\widetilde{Q}}(\delta\xi,\delta{f},\delta{S})
+\widetilde{C}_2\Vert  \delta S\Vert_{\mathbb{D}_{T\wedge\tau}(\widetilde{Q},p)}^{1/2}\sqrt{\sum_{i=1}^2 \Delta_{\widetilde{Q}}(\xi^{(i)},f^{(i)},(S^{(i)})^+)}.\end{align*}
Here $\Delta_{\widetilde{Q}}(\xi^{(i)},f^{(i)},(S^{(i)})^+)$ for $i=1,2$ and  $\Delta_{\widetilde{Q}}(\delta\xi,\delta{f},\delta{S})$ are given via (\ref{Delta(xi,f,S)}), while $(\delta Y^{\mathbb G},\delta Z^{\mathbb G},\delta M^{\mathbb G}, \delta{K}^{\mathbb{G}})$ and $(\delta\xi,\delta{f}, \delta{S})$ are defined in (\ref{Delta4Solution}).   \end{theorem}
\begin{proof}
By applying Lemma \ref{technicallemma1}-(a) to $Y:=\delta{Y}^{\mathbb G}$, we deduce that 
\begin{equation}\label{Control4deltaYGInfinity}
E\left[\sup_{0\leq t\leq T}\widetilde{\cal E}_t\vert\delta Y^{\mathbb{G}}_{t}\vert^p\right]\leq G_0^{-1} E^{\widetilde{Q}}\left[\sup_{0\leq t\leq T}\vert\delta Y^{\mathbb{G}}_{t}\vert^p\right].\end{equation}
An application of Lemma \ref{technicallemma1}-(b) to $K=\int_0^{\cdot }(\delta{Z}^{\mathbb G}_s)^2ds +[\delta{M}^{\mathbb G},\delta{M}^{\mathbb G}]$ with $a=2/p$ implies that  
\begin{equation}\label{Control4deltaZGInfinity}
\begin{split}
&E\left[\left(\int_0^{T\wedge\tau}(\widetilde{\cal E}_{s-})^{2/p}(\delta Z^{\mathbb{G}}_s)^2 ds+\int_0^{T\wedge\tau}(\widetilde{\cal E}_{s-})^{2/p}d[\delta{M}^{\mathbb G},\delta{M}^{\mathbb G}]_s\right)^{p/2}\right]\\
&\leq\kappa E^{\widetilde{Q}}\left [\left(\int_0^{T\wedge\tau}(\delta Z^{\mathbb{G}}_s)^2 ds+[\delta{M}^{\mathbb G},\delta{M}^{\mathbb G}]_{T\wedge\tau}\right)^{{{p}\over{2}}}+\sum_{0\leq{t}\leq {T\wedge\tau}}
{\widetilde{G}_t}\vert\Delta(\delta{M}^{\mathbb G})_t\vert^p\right].\end{split}
\end{equation}
Then by using $\Delta(\delta{M}^{\mathbb G})=(\delta{h}-\delta{Y}^{\mathbb F}{\widetilde{\cal E}}^{-1})\Delta N^{\mathbb G}=:H\Delta N^{\mathbb G}$ and ${\widetilde{\cal E}}^{-1}\delta{Y}^{\mathbb F}=\delta{Y}^{\mathbb G}$ on ${\Rbrack0,\tau\Lbrack}$ --see  Theorem \ref{abcde} --, and by mimicking parts 3 and 4 in the proof of Lemma \ref{technicallemma1}, we derive
\begin{equation}\label{Control4deltaMGInfinity}
\begin{split}
 E^{\widetilde{Q}}\left [\sum_{0\leq{t}\leq {T\wedge\tau}}{\widetilde{G}_t}\vert\Delta(\delta{M}^{\mathbb G})_t\vert^p\right]&\leq  E^{\widetilde{Q}}\left [\widetilde{G}\vert{H}\vert^p\is\mbox{Var}(N^{\mathbb G})_T\right]=2E^{\widetilde{Q}}\left [\vert{H}\vert^pI_{\Rbrack0,\tau\Lbrack}\is{D}^{o,\mathbb F}_T\right]\\
 &\leq2^p E^{\widetilde{Q}}\left [(\vert{\delta{h}}\vert^p+\vert\delta{Y}^{\mathbb G}\vert^p)I_{\Rbrack0,\tau\Lbrack}\is {D}^{o,\mathbb F}_T\right]\\
 &\leq 2^pE^{\widetilde{Q}}\left[\vert\delta\xi\vert^p+ \sup_{0\leq{t}\leq{T\wedge\tau}}\vert \delta{Y}^{\mathbb G}\vert^p\right].\end{split}
\end{equation}
The last inequality follows from  Lemma \ref{Lemma4.11}-(b). Therefore, by combining (\ref{Control4deltaYGInfinity}), (\ref{Control4deltaZGInfinity}), (\ref{Control4deltaMGInfinity}) and Theorem \ref{EstimatesUnderQtilde1}, the proof of the theorem follows immediately. This ends the proof of the theorem.
\end{proof}     
 {\it{Our second step}} in solving (\ref{RBSDEGinfinite}) relies on the following lemma, and focuses on letting $T$ to go to infinity in order to get rid-off $\widetilde{Q}$ in the norms of the data-triplet. This {\it naturally} leads us to define the appropriate norm and space for the data-triplet of this RBSDE (\ref{RBSDEGinfinite}).
     \begin{lemma}\label{ExpecationQtilde2P} Let $X$ be a non-negative and $\mathbb F$-optional process such that $X_0=0$ and $X/{\cal E}(G_{-}^{-1}\cdot m)$ is bounded. Then
     \begin{eqnarray}
     \lim_{T\to\infty} E^{\widetilde Q}[X_{T\wedge\tau}]=G_0\Vert X\Vert_{L^1(P\otimes V^{\mathbb F})}:=G_0 E\left[\int_0^{\infty} X_sdV_s^{\mathbb F}\right].\end{eqnarray}
     \end{lemma}
The proof of this lemma is relegated to Appendix \ref{Appendix4Proofs}. It is clear that this lemma allows us to take the limit of $\widetilde Q$-expectations, under some conditions. More importantly, on the one hand, this leads {\it naturally} to the space $$L^p(\Omega\times [0,+\infty), {\cal{F}}\otimes{\cal{B}}(\mathbb R^+),P\otimes{V}^{\mathbb F})$$ for the data-triplet $(f,h,S)$, endowed with its norm defined by
\begin{equation}\label{Lp(PVF)}
\Vert{X}\Vert_{L^p(P\otimes{V}^{\mathbb{F}})}^p:=E\left[\int_0^{\infty}\vert{X}_t\vert^p dV^{\mathbb{F}}_t\right],\quad\mbox{ for any $ {\cal{F}}\otimes{\cal{B}}(\mathbb R^+)$-measurable $X$}.\end{equation}
On the other hand, the pair $(Y^{\mathbb{G}},Z^{\mathbb{G}})$ in the solution of (\ref{RBSDEG}) belongs to $\widetilde{\mathbb{D}}_{\sigma}(P,p)\times \widetilde{\mathbb{S}}_{\sigma}(P,p)$, with $\sigma=T\wedge\tau$. The two spaces, which appear {\it naturally} in our analysis, are given as follows: $(Y,Z)$ belongs to $\widetilde{\mathbb{D}}_{\sigma}(P,p)\times \widetilde{\mathbb{S}}_{\sigma}(P,p)$ iff $(Y\sqrt[p]{{\widetilde{\cal E}}},Z\sqrt[p]{{\widetilde{\cal E}}_{-}})\in  {\mathbb{D}}_{\sigma}(P,p)\times{\mathbb{S}}_{\sigma}(P,p)$, and  
\begin{equation}\label{DtildeSpace}
 \Vert{Y}\Vert_{ \widetilde{\mathbb{D}}_{\sigma}(P,p)}:= \Vert{Y}\sqrt[p]{\widetilde{\cal E}}\Vert_{ {\mathbb{D}}_{\sigma}(P,p)}\quad\mbox{and}\quad \Vert{Z}\Vert_{ \widetilde{\mathbb{S}}_{\sigma}(P,p)}:= \Vert{Z}\sqrt[p]{\widetilde{\cal E}_{-}}\Vert_{ {\mathbb{S}}_{\sigma}(P,p)}.
 \end{equation}
  Similarly, for the pair  $(K^{\mathbb{G}},M^{\mathbb{G}})$ in the solution of (\ref{RBSDEG}), we take the norm under $P$ of the ``discounted'' pair $(\sqrt[p]{\widetilde{\cal E}_{-}}\is{K}^{\mathbb{G}},\sqrt[p]{\widetilde{\cal E}_{-}}\is{M}^{\mathbb{G}})$ instead. Below, we elaborate our principal result of this subsection.

\begin{theorem}\label{EstimateInfinite} Let $p\in (1,+\infty)$, and suppose $G>0$ and $(f, S, h)$ satisfies 
\begin{eqnarray}\label{MainAssumption4InfiniteHorizon}
\Delta_{P\otimes{V}^{\mathbb F}}(f,h,S):=\norm{\int_0^{\cdot}\vert f(s)\vert ds+\vert{h}\vert+\sup_{0\leq{u}\leq\cdot}\vert{S}_u\vert}_{L^p(P\otimes{V}^{\mathbb F})}<\infty.
\end{eqnarray}
Then the following assertions hold.\\
{\rm{(a)}} The RBSDE (\ref{RBSDEGinfinite}) admits a unique solution $\left(Y^{\mathbb{G}},Z^{\mathbb{G}},K^{\mathbb{G}},M^{\mathbb{G}}\right)$. \\
{\rm{(b)}} There exists a positive constant $C$, that depends on $p$ only, such that 
\begin{eqnarray*}
\norm{\norm{\left(\sqrt[p]{\widetilde{\cal E}}{Y}^{\mathbb{G}},\sqrt[p]{\widetilde{\cal E}_{-}}{Z}^{\mathbb{G}},\sqrt[p]{\widetilde{\cal E}_{-}}\is{M}^{\mathbb{G}} ,\sqrt[p]{\widetilde{\cal E}_{-}}\is{K}^{\mathbb{G}}\right)}}_{(P,\mathbb{G},p)}\leq C\Delta_{P\otimes{V}^{\mathbb F}}(f,h,S^+).
\end{eqnarray*}
{\rm{(c)}} Consider two triplets $(f^{(i)}, S^{(i)}, h^{(i)})$, $i=1,2$ satisfying (\ref{MainAssumption4InfiniteHorizon}). If the solution to (\ref{RBSDEGinfinite}) associated with $(f^{(i)}, S^{(i)}, h^{(i)})$ is $\left(Y^{\mathbb{G},i},Z^{\mathbb{G},i},K^{\mathbb{G},i}, M^{\mathbb{G},i}\right)$, then there exist positive $C_1$ and $C_2$ that depend on $p$ only such that 
 \begin{eqnarray*}
 \begin{split}
 &\norm{\norm{\left(\sqrt[p]{\widetilde{\cal E}}\delta{Y}^{\mathbb{G}},\sqrt[p]{\widetilde{\cal E}_{-}}\delta{Z}^{\mathbb{G}},\sqrt[p]{\widetilde{\cal E}_{-}}\is\delta{M}^{\mathbb{G}} ,0\right)}}_{(P,\mathbb{G},p)}\\
&\leq C_1 \Delta_{P\otimes{V}^{\mathbb F}}(\delta{f},\delta{h},\delta{S})+C_2\sqrt{ \Delta_{P\otimes{V}^{\mathbb F}}(0,0,\delta{S})\sum_{i=1}^2 \Delta_{P\otimes{V}^{\mathbb F}}(f^{(i)},h^{(i)},(S^{(i)})^+)}.\end{split}\end{eqnarray*}
Here $\delta{Y}^{\mathbb{G}},\delta{Z}^{\mathbb{G}}, \delta{M}^{\mathbb{G}}, \delta{K}^{\mathbb{G}}$ and $\delta{S}$ are given by (\ref{Delta4Solution}), and 
\begin{eqnarray}\label{processesDelta}
\delta{h}:=h^{(1)}-h^{(2)} ,\ \delta{f}:={f}^{(1)}-{f}^{(2)},\ F^{(i)}:=\int_0^{\cdot}\vert{f}^{(i)}_s\vert{d}s,\ i=1,2.\end{eqnarray}
{\rm{(d)}} Let  ${\widetilde{V}}^{(1/p)}$ be defined in (\ref{Vepsilon}). Then, there exists a unique $L^p(P,\mathbb G)$-solution to 
\begin{equation}\label{EquivalentRBSDE}
 \begin{split}
& dY=-Y\sqrt[p]{{{\widetilde{G}}\over{G}}}I_{\Rbrack0,\tau\Rbrack}d{\widetilde{V}}^{(1/p)}-f(t)\sqrt[p]{\widetilde{\cal E}_{-}}d(t\wedge\tau)-dK-dM+ZdW^{\tau},\\
 &Y_{\tau}=\xi\sqrt[p]{\widetilde{\cal E}_{\tau}},\quad Y\geq S\sqrt[p]{\widetilde{\cal E}}\ \mbox{on}\ \Lbrack0,\tau\Lbrack,\quad \int_0^{\tau}(Y_{u-}- S_{u-}\sqrt[p]{\widetilde{\cal E}_{u-}})dK_u=0,
\end{split}
 \end{equation}
denoted by  $\left(\widetilde{Y}^{\mathbb{G}},\widetilde{Z}^{\mathbb{G}},\widetilde{K}^{\mathbb{G}},\widetilde{M}^{\mathbb{G}}\right)$, and which satisfies
  \begin{equation}\label{Relationship}
  \left(\widetilde{Y}^{\mathbb{G}},\widetilde{Z}^{\mathbb{G}},\widetilde{K}^{\mathbb{G}},\widetilde{M}^{\mathbb{G}}\right)=\left(Y^{\mathbb{G}}\sqrt[p]{\widetilde{\cal E}},Z^{\mathbb{G}}\sqrt[p]{\widetilde{\cal E}_{-}},\sqrt[p]{\widetilde{\cal E}_{-}}\is{K}^{\mathbb{G}},\sqrt[p]{\widetilde{\cal E}_{-}}\is{M}^{\mathbb{G}}\right).\end{equation}
\end{theorem}
The proof of the theorem is technically involved. Hence, for the sake of a better exposition, we relegate its proof to Subsection \ref{Subsection4proofTheorem5.5}. \\
It is worth mentioning that Theorem \ref{EstimateInfinite} requires the stronger condition$\Vert\sup_{0\leq{u}\leq\cdot}\vert{S}_u\vert \Vert_{L^p(P\otimes{V}^{\mathbb F})}<\infty$ compared to $\Vert\sup_{0\leq{u}\leq\cdot}{S}_u^+\Vert_{L^p(P\otimes{V}^{\mathbb F})}<\infty$, see (\ref{MainAssumption4InfiniteHorizon}). Importantly, this assumption (\ref{MainAssumption4InfiniteHorizon}) reduces the generality of the theorem and similar results for BSDEs can not be derived from this theorem in contrast to previous theorems. However, our method remains valid for BSDEs, and one can easily adapt our proof to the BSDE setting directly by just ignoring the process $S$ and putting $K^{\mathbb G}=0$ throughout the proof, as they are both irrelevant for the BSDE case. This proves the following 
\begin{theorem}\label{BSDEInfinite} Let $p\in (1,+\infty)$ and $F$ be given by (\ref{MainAssumption4InfiniteHorizon}).  Suppose $G>0$ and $(f, h)$ satisfies 
$\Vert{F}+\vert{h}\vert \Vert_{L^p(P\otimes{V}^{\mathbb F})}<+\infty.$ Then the following hold.\\
{\rm{(a)}} The following BSDE
\begin{eqnarray}\label{BSDEinfinite}
 dY=-f(t)d(t\wedge\tau)-dM+ZdW^{\tau},\quad Y_{\tau}=\xi.
\end{eqnarray}
  admits a unique solution $\left(Y^{\mathbb{G}},Z^{\mathbb{G}},M^{\mathbb{G}}\right)$. \\
{\rm{(b)}} There exists a positive constant $C$, that depends on $p$ only, such that 
\begin{eqnarray*}
\Vert{Y}^{{\mathbb{G}}}\Vert_{\widetilde{\mathbb{D}}_{\tau}(P,p)}+\Vert{Z}^{\mathbb{G}}\Vert_{\widetilde{\mathbb{S}}_{\tau}(P,p)}
+\Vert\sqrt[p]{\widetilde{\cal E}_{-}}\is {M}^{\mathbb{G}}\Vert_{{\cal{M}}^p(P)}\nonumber\leq C \Vert{F}+\vert{h}\vert\Vert_{L^p(P\otimes V^{\mathbb F})}.\end{eqnarray*}
{\rm{(c)}} Consider $(f^{(i)}, h^{(i)})$, $i=1,2$, satisfying $\Vert{F^{(i)}}+\vert{h^{(i)}}\vert \Vert_{L^p(P\otimes{V}^{\mathbb F})}<+\infty.$ If $\left(Y^{\mathbb{G},i},Z^{\mathbb{G},i}, M^{\mathbb{G},i}\right)$ denotes the solution to (\ref{BSDEinfinite}) associated with   $(f^{(i)}, h^{(i)})$ for each $i=1,2$, then there exist positive $C_1$ and $C_2$ that depend on $p$ only such that 
 \begin{eqnarray*}
 \Vert\delta Y^{{\mathbb{G}}}\Vert_{\widetilde{\mathbb{D}}_{\tau}(P,p)}+\Vert \delta Z^{\mathbb{G}}\Vert_{\widetilde{\mathbb{S}}_{\tau}(P,p)}+\Vert\sqrt[p]{\widetilde{\cal E}_{-}}\is\delta M^{\mathbb{G}}\Vert_{{\cal{M}}^p(P)}\leq C_1 \Vert \vert\delta h\vert+\vert{\delta F}\vert\Vert_{L^p(P\otimes V^{\mathbb F})}.
\end{eqnarray*}
Here $\delta{Y}^{\mathbb{G}},\delta{Z}^{\mathbb{G}}$ and $ \delta{M}^{\mathbb{G}}$ are given by (\ref{Delta4Solution}),  $ {F}^{(i)}$ is defined via (\ref{MainAssumption4InfiniteHorizon}), and 
\begin{eqnarray*}
\delta{h}:=h^{(1)}-h^{(2)} ,\ \delta{F}:=\int_0^{\cdot}\vert{f}^{(1)}_s-{f}^{(2)}_s\vert{d}s.\end{eqnarray*}
{\rm{(d)}} Let  ${\widetilde{V}}^{(1/p)}$ be defined in (\ref{Vepsilon}). Then, there exists a unique $L^p(P,\mathbb G)$-solution to 
\begin{eqnarray*}
 dY=-Y\sqrt[p]{{{\widetilde{G}}\over{G}}}I_{\Rbrack0,\tau\Rbrack}d{\widetilde{V}}^{(1/p)}-\sqrt[p]{\widetilde{\cal E}_{-}}f(t)d(t\wedge\tau)-dM+ZdW^{\tau},\ Y_{\tau}=\sqrt[p]{\widetilde{\cal E}_{\tau}}\xi,
 \end{eqnarray*}
denoted by  $\left(\widetilde{Y}^{\mathbb{G}},\widetilde{Z}^{\mathbb{G}},\widetilde{M}^{\mathbb{G}}\right)$ satisfying $$
  \left(\widetilde{Y}^{\mathbb{G}},\widetilde{Z}^{\mathbb{G}},\widetilde{M}^{\mathbb{G}}\right)=\left(Y^{\mathbb{G}}\sqrt[p]{\widetilde{\cal E}},Z^{\mathbb{G}}\sqrt[p]{\widetilde{\cal E}_{-}},\sqrt[p]{\widetilde{\cal E}_{-}}\is{M}^{\mathbb{G}}\right).$$
\end{theorem}
\subsection{Relationship to RBSDE under $\mathbb F$}
Hereto, we establish the $\mathbb{F}$-RBSDE counter part of (\ref{RBSDEGinfinite}).
   
\begin{theorem}\label{Relationship4InfiniteBSDE}  Let $p\in (1,+\infty)$,  and $ {\widetilde{\cal E}}$ and $(f^{\mathbb{F}},S^{\mathbb{F}})$ be given by (\ref{EpsilonTilde}) and (\ref{ProcessVFandXiF}) respectively.  Suppose $G>0$, $(f, S, h)$ satisfies (\ref{MainAssumption4InfiniteHorizon}), 
\begin{eqnarray}\label{MainAssumption4InfiniteHorizonBIS}
E\left[\left(F_{\infty}{\widetilde{\cal E}}_{\infty}\right)^p\right]<+\infty,\quad\mbox{where}\ F_{\infty}:=\int_0^{\infty}\vert{f}_s\vert ds,\end{eqnarray}
and
\begin{equation}\label{Condition4S(F)}
\limsup_{t\longrightarrow\infty}\left(S^{\mathbb{F}}_t\right)^+=0,\quad P\mbox{-a.s..}
\end{equation}
Then the following assertions hold.\\
{\rm{(a)}} The following RBSDE, under $\mathbb F$, generated by the triplet $ \left(f^{\mathbb{F}},S^{\mathbb{F}},h\right)$
\begin{equation}\label{RBSDEFinfinite}
\begin{split}
&Y_{t}= \displaystyle\int_{t}^{\infty}f^{\mathbb{F}}(s)ds+\int_{t}^{\infty}h_{s}dV^{\mathbb{F}}_{s}+K_{\infty}-K_{t}-\int_{t}^{\infty}Z_{s}dW_{s},\\
&Y_{t}\geq S_{t}^{\mathbb{F}},\quad
 \displaystyle{E}\left[\int_{0}^{\infty}(Y_{t-}-S_{t-}^{\mathbb{F}})dK_{t}\right]=0 ,
\end{split}
\end{equation}
has a unique  $L^p(P,\mathbb F)$-solution $(Y^{\mathbb F},  Z^{\mathbb F}, K^{\mathbb F})$ satisfying
\begin{equation}\label{SnellRepreentation4BSDE(F,infinity)}
Y_t^{\mathbb F}=\esssup_{\sigma\in{\cal T}_t^{\infty}(\mathbb F)}E\left[\int_t^{\sigma} f^{\mathbb F}_s ds+\int_t^{\sigma}h_s dV^{\mathbb F}_s+S^{\mathbb F}_{\sigma}I_{\{\sigma<+\infty\}}\ \big|{\cal F}_t\right],  \end{equation}
and there exists a universal constant $C_1\in (0,\infty)$ such that 
\begin{equation}\label{Estimate(F,infinity)}
\begin{split}
\Vert{Y}^{\mathbb F}\Vert_{\mathbb{D}_{\infty}(P,p)}+\Vert {Z}^{\mathbb F}\Vert_{\mathbb{S}_{\infty}(P,p)}+\Vert{K}^{\mathbb F}_{\infty}\Vert_{\mathbb{L}^p(P)}\leq C_1\norm{\int_0^{\infty} \vert{f}^{\mathbb{F}}_s\vert{d}s+\int_{0}^{\infty}\vert{h}_s\vert{d}V^{\mathbb{F}}_{s}+\sup_{0\leq t}(S^{\mathbb{F}}_t)^+}_{\mathbb{L}^p(P)}.\end{split}
\end{equation}
{\rm{(b)}} The solution to (\ref{RBSDEGinfinite}), denoted by $\left(Y^{\mathbb{G}},Z^{\mathbb{G}},K^{\mathbb{G}},M^{\mathbb{G}}\right)$, satisfies
\begin{equation}\label{firstrelationInfnite}
\begin{split}
   Y^{\mathbb{G}}= {{Y^{\mathbb{F}}}\over{\widetilde{\cal E}}}I_{\Lbrack0,\tau\Lbrack}+\xi{I}_{\Lbrack\tau,+\infty\Lbrack},\ 
  Z^{\mathbb{G}}={{Z^{\mathbb{F}}}\over{\widetilde{\cal E}_{-}}} I_{\Rbrack0,\tau\Rbrack},\quad K^{\mathbb{G}}={{I_{\Rbrack0,\tau\Rbrack}}\over{\widetilde{\cal E}_{-}}}\is{K} ^{\mathbb{F}},\ \mbox{and}\ 
      M^{\mathbb{G}}=\left(h-{{Y^{\mathbb{F}}}\over{\widetilde{\cal E}}}\right)\is{ N}^{\mathbb{G}}.
      \end{split}
       \end{equation}
\end{theorem}
The RBDSE (\ref{RBSDEFinfinite}) can reformulated by saying that the triplet $(Y,Z,K)$ is a solution to (\ref{RBSDEFinfinite}) if the following two properties hold:\\
a) $Y$ is RCLL and $\mathbb{F}$-adapted such that $Y_T$ converges to zero almost surely when $T$ goes to infinity, $K$ is RCLL nondecreasing and $\mathbb{F}$-predictable,\\
b) $Z$ is $\mathbb{F}$-progressive and $W$-integrable, and for any $T\in(0,\infty)$ and $0\leq t\leq T$,
\begin{equation*}
\begin{split}
&Y_{t}=Y_T+ \displaystyle\int_{t}^{T}f^{\mathbb{F}}(s)ds+\int_{t}^{T}h_{s}dV^{\mathbb{F}}_{s}+K_{T}-K_{t}-\int_{t}^{T}Z_{s}dW_{s},\\
&Y_{u}\geq S_{u}^{\mathbb{F}},\quad{u}\geq 0,\quad 
 \displaystyle{E}\left[\int_{0}^{\infty}(Y_{u-}-S_{u-}^{\mathbb{F}})dK_u\right]=0 ,
\end{split}
\end{equation*}
For similar definition in either the case of BSDE with unbounded horizon but finite or the case of RBSDE with infinite horizon, we refer to \cite[Section 3]{Klimsiak0}  and \cite{Hamadane0} respectively. Thus, from this perspective and for the linear case, our RBSDE generalizes \cite{Klimsiak0} to the cases where the horizon is infinite and having reflection barrier, while it generalizes \cite{Hamadane0} by letting $S$ and $K$ to be arbitrary RCLL.  However, the nonlinear setting of \cite{Klimsiak0}  and \cite{Hamadane0}, in which the driver $f$ depends on $Y$ and/or $Z$, can be found in \cite{Alsheyab,Choulli6}. 
  \begin{remark} {\rm{(a)}} It is clear that, in general, the existence of an $L^p(P,\mathbb{F})$-solution to (\ref{RBSDEFinfinite}) requires stronger assumptions than the existence of $L^p(P,\mathbb{G})$-solution to (\ref{RBSDEGinfinite}).\\
   {\rm{(b)}} It is worth mentioning that (\ref{Condition4S(F)}) is necessary for the existence of a solution to (\ref{RBSDEFinfinite}). In fact, the conditions $Y^{\mathbb{F}}_t\geq S^{\mathbb{F}}_t$ for any $t\geq 0$ and $\lim_{t\longrightarrow\infty}Y^{\mathbb{F}}_t=0$ $P$-a.s. yield  (\ref{Condition4S(F)}) immediately.\\
 {\rm{(c)}} Similar theorem for the BSDE setting can be easily derived, while in this case the condition (\ref{Condition4S(F)}) becomes obvious as $S\equiv-\infty$. For more details, see Theorem \ref{BSDEInfinite} and the discussions before it.
  \end{remark}
\subsection{Proof of Theorems \ref{EstimateInfinite} and \ref{Relationship4InfiniteBSDE} }\label{Subsection4proofTheorem5.5}
\begin{proof}[Proof of Theorem \ref{Relationship4InfiniteBSDE}] On the one hand, remark that, due to the assumptions  (\ref{MainAssumption4InfiniteHorizon})  and (\ref{MainAssumption4InfiniteHorizonBIS}), both random variables $\int_0^{\infty}\vert f^{\mathbb F}_s\vert ds$ and $\int_0^{\infty}\vert{h}_s\vert dV^{\mathbb F}_s$ belong to $L^p(P)$. Indeed, in virtue of $ (V^{\mathbb F}_{\infty})^{p-1}\leq 1$, this fact follows from 
\begin{eqnarray*}
\int_0^{\infty}\vert f^{\mathbb F}_s\vert ds={\widetilde{\cal E}}_{\infty}F_{\infty}+\int_0^{\infty} F_s dV^{\mathbb F}_s,\ \mbox{and}\ \left(\int_0^{\infty}\vert{h}_s\vert dV^{\mathbb F}_s\right)^p\leq\int_0^{\infty}\vert{h}_s\vert^p dV^{\mathbb F}_s .
\end{eqnarray*}
The rest of this proof is divided into three parts. The first part proves that  (\ref{RBSDEFinfinite}) has in fact an $L^p(P,\mathbb F)$-solution satisfying  (\ref{Estimate(F,infinity)}). The second part  proves that when a solution to (\ref{RBSDEFinfinite}) --say $(Y,Z,K)$-- exists, then $Y$ is of class D and coincides with $\overline{Y}$ given by 
\begin{equation}\label{Ybar}
\overline{Y}_t:=\esssup_{\sigma\in{\cal T}_t^{\infty}(\mathbb F)}E\left[\int_t^{\sigma} f^{\mathbb F}_s ds+\int_t^{\sigma}h_s dV^{\mathbb F}_s+S^{\mathbb F}_{\sigma}I_{\{\sigma<+\infty\}}\ \big|{\cal F}_t\right].   \end{equation}
Hence, this combined with the uniqueness of the Doob-Meyer decomposition imply the uniqueness of the solution to  (\ref{RBSDEFinfinite}). These two parts (parts 1 and 2) prove assertion (a), while assertion (b) is proved in the third part.\\
{\bf Part 1.} By using $-(f^{\mathbb{F}}(s))^-\leq f^{\mathbb{F}}(s)\leq (f^{\mathbb{F}}(s))^+$, $-(h_s)^-\leq h_s\leq(h_s)^+$, $ S_u^{\mathbb{F}}\leq \sup_{u\geq 0}( S_u^{\mathbb{F}})^+$ and $E\left[\int_t^{\infty} f^{\mathbb F}_s ds+\int_t^{\infty}h_s dV^{\mathbb F}_s\ \big|{\cal F}_t\right]\leq
\overline{Y}_t$, we get 
\begin{equation*}
\vert \overline{Y}_t\vert\leq{E}\left[\int_0^{\infty} \vert{f}^{\mathbb F}_s\vert{ d}s+\int_0^{\infty}\vert{h}_s\vert{ d}V^{\mathbb F}_s+\sup_{u\geq0}({S}^{\mathbb F}_u)^+\ \big|{\cal F}_t\right].
\end{equation*}
Then the Doob's inequality yields
\begin{equation}\label{YbarSupremum}
\Vert\overline{Y}\Vert_{{\mathbb{D}_{\infty}(P,p)}}\leq q \displaystyle\Big\Vert \int_0^{\infty} \vert{f}^{\mathbb F}_s\vert{ d}s+\int_0^{\infty}\vert{h}_s\vert{ d}V^{\mathbb F}_s+\sup_{u\geq 0}(S^{\mathbb F}_u)^+ \Big\Vert_{{\mathbb{L}}^p(P)}
\end{equation}
Remark that $\overline{X}:=\overline{Y}+\int_{0}^{\cdot}f^{\mathbb{F}}(s)ds+\int_{0}^{\cdot}h_s dV^{\mathbb{F}}_s$ is the Snell envelop for the reward process $R:=S^{\mathbb{F}}+\int_{0}^{\cdot}f^{\mathbb{F}}(s)ds+\int_{0}^{\cdot}h_s dV^{\mathbb{F}}_s$, which is RCLL  and of class D. Hence, thanks to \cite[Remark 23-(c), Appendice 1]{DellacherieMeyer80}, and the martingale representation theorem, we deduce the existence of unique pair  $(\overline{Z}, \overline{K})\in L^1_{loc}(W,P)\times{\cal{A}}^+$, such that $\overline{K}$ is predictable, $Z\is{W}\in{\cal{M}}$, and   
\begin{equation}\label{Decompo4Xbar}\begin{split}
&\overline{Y}=\overline{Y}_0-\int_{0}^{\cdot}f^{\mathbb{F}}(s)ds-\int_{0}^{\cdot}h_s dV^{\mathbb{F}}_s+\overline{Z}\is{W}-\overline{K},\\
&\overline{Y}\geq S^{\mathbb{F}},\quad  \int_0^{\infty}I_{\{\overline{Y}_{s-}>S_{s-}^{\mathbb{F}}\}}d\overline{K}_s=\int_0^{\infty}I_{\{\overline{X}_{s-}>R_{s-}\}}d\overline{K}_s=0,\quad P\mbox{-a.s.}. 
\end{split}\end{equation}
Thus, on the one hand, it is clear that the triplet $(\overline{Y},\overline{Z},\overline{K})$ is a solution to (\ref{RBSDEFinfinite}). On the other hand, a direct application of \cite[Lemma 1.9]{Stricker} to $-\overline{X}$, which is a submartingale, we deduce there exists a universal constant $C_p$ (i.e. depends on $p$ only), such that 
\begin{equation*}\label{Control4Z(m)}
\begin{split}
\Vert \overline{Z}\Vert_{{\mathbb{S}_{\infty}(P,p)}}+\Vert\overline{K}_{\infty}\Vert_{L^p(P)}&\leq C_p\Vert \overline{X}\Vert_{{\mathbb{D}_{\infty}(P,p)}}\\
&\leq C_p\Vert \overline{Y}\Vert_{{\mathbb{D}_{\infty}(P,p)}}+C_p\Big\Vert\int_{0}^{\infty}\vert{f}^{\mathbb{F}}(s)\vert{d}s+\int_{0}^{\infty}\vert{h}_s\vert{d}V^{\mathbb{F}}_s\Big\Vert_{L^p(P)}.
\end{split}
\end{equation*}
Therefore, by combining this latter inequality and (\ref{YbarSupremum}), we get (\ref{Estimate(F,infinity)}) for $(\overline{Y},\overline{Z},\overline{K})$, and the first part is complete.\\
{\bf Part 2.}  In this part, we assume there exists an $L^p(R_0)$-solution to (\ref{RBSDEFinfinite}), denoted by  $(Y,Z, K)$, and prove that $Y$ coincides with $\overline{Y}$. Then, $Y$ is of class D, and thanks again to \cite[Lemma 1.9]{Stricker} again, we deduce that $Z\is{W}\in{\cal{M}}^p$ and hence 
\begin{equation}
\begin{split}
Y_{t}&\geq{E}\left[Y_{\sigma}+\int_{t}^{\sigma}f^{\mathbb F}_s{d}s+\int_{t}^{\sigma}h_s dV^{\mathbb F}_s\big|{\cal{F}}_{t}\right]={E}\left[Y_{\sigma}I_{\{\sigma<\infty\}}+\int_{t}^{\sigma}f^{\mathbb F}_s{d}s+\int_{t}^{\sigma}h_s dV^{\mathbb F}_s\big|{\cal{F}}_{t}\right]\\
&\geq {E}\left[S_{\sigma}I_{\{\sigma<\infty\}}+\int_{t}^{\sigma}f^{\mathbb F}_s{d}s+\int_{t}^{\sigma}h_s dV^{\mathbb F}_s\big|{\cal{F}}_{t}\right].
\end{split}
\end{equation}
Thus, by taking the essential supremum over $\sigma\in{\cal{T}}_t^{\infty}$, the latter inequality implies that $Y_t\geq \overline{Y}_t$ $P$-a.s.. To prove the reverse, we consider the following stopping times
 \begin{equation*}
   \theta_{n}:=\inf\left\{u\geq t:\quad Y_{u}<S_u^{\mathbb{F}}+\frac{1}{n}\right\}\in\mathcal{J}_{t}^{\infty}(\mathbb{F}),\quad n\geq 1,
 \end{equation*}
  and remark that
  \begin{eqnarray*}
 Y-S\geq\frac{1}{n}\quad  \mbox{on }\quad  \Lbrack t,\theta_{n}\Lbrack,\quad\mbox{and}\quad 
Y_{-}-S_{-}\geq\frac{1}{n}\quad  \mbox{on }\quad \Rbrack t,\theta_{n}\Rbrack,\quad (S_{\theta_n})-Y_{\theta_n})I_{\{\theta_n<\infty\}}\geq -{1\over{n}}.
   \end{eqnarray*}
As a result, ${ K}_{\theta_n}-K_t=0$ $P$-a.s., and by using
\begin{equation*}\label{Y(n,k)}
Y_{t}=Y_{\theta_n}+\int_{t}^{\theta_n}f^{\mathbb F}_s{d}s+\int_{t}^{\theta_n}h_s dV^{\mathbb F}_s+K_{\theta_n}-K_{t} -\int_{t}^{\theta_n}Z_s{dW_s},\end{equation*}
we derive
     \begin{equation*}      \begin{split}
    \overline{Y}_t&\geq E\left[\int_{t}^{\theta_{n}}f^{\mathbb{F}}(s)ds+\int_{t}^{\theta_{n}}h_s dV^{\mathbb{F}}_s+S_{\theta_{n}}1_{\{\theta_{n}<\infty\}}\ \Big|\ \mathcal{F}_{t}\right]  \\
    &   =    Y_{t}+E\left[(S_{\theta_{n}}-Y_{\theta_{n}})1_{\{\theta_{n}<\infty\}} \Big|\ \mathcal{F}_{t}\right]\geq Y_{t}-\frac{1}{n}P(\theta_{n}<\infty|\mathcal{F}_{t}).
         \end{split} \end{equation*} 
  Therefore, this yields $\overline{Y}_t\geq{Y}_t$, and  the equality $\overline{Y}_t={Y}_t$ follows. This ends the second part and completes the proof of assertion (a). \\
{\bf Part 3.} Here we prove assertion (b). Remark that, in virtue of Theorem \ref{EstimateInfinite},  the RBSDE (\ref{RBSDEGinfinite}) has at most one solution. Therefore, it is enough to prove that $(\widehat{Y}, \widehat{Z}, \widehat{K},\widehat{M})$ given by 
\begin{equation*}
\begin{split} \widehat{Y}:= \frac{Y^{\mathbb{F}}}{{\widetilde{\cal E}}}I_{\Lbrack0,\tau\Lbrack}+\xi{I}_{\Lbrack\tau,+\infty\Lbrack},\ 
\widehat{Z}:=\frac{Z^{\mathbb{F}}}{{\widetilde{\cal E}}_{-}} I_{\Rbrack0,\tau\Rbrack},\quad   \widehat{K}:=\frac{1}{{\widetilde{\cal E}}_{-}}\is(K ^{\mathbb{F}})^{\tau},\ \mbox{and}\ 
   \widehat{M}:=\left(h-\frac{Y^{\mathbb{F}}}{\widetilde{\cal E}}\right)\is{ N}^{\mathbb{G}}, \end{split} \end{equation*}
   is indeed a solution to  (\ref{RBSDEGinfinite}). To this end, we put 
$\widehat{\Gamma}:={Y}^{\mathbb F}/{\widetilde{\cal E}},$ and on the one hand we remark that  
\begin{equation}\label{YGGammainfinite}
\widehat{Y} =\widehat{\Gamma}I_{\Lbrack0,\tau\Lbrack}+  h_{\tau} I_{\Lbrack\tau,+\infty\Lbrack}=\widehat{\Gamma}^{\tau} +(h-\widehat{\Gamma})\is{ D}.\end{equation}
On the other hand, by combining It\^o applied to $\widehat{\Gamma}$, (\ref{RBSDEFinfinite}) that the triplet $(Y^{\mathbb F}, Z^{\mathbb{F}}, K^{\mathbb{F}})$ satisfies, ${\widetilde{\cal E}}^{-1}={\cal E}(G^{-1}\is{ D}^{o,\mathbb F})$, ${\widetilde{\cal E}}={\widetilde{\cal E}}_{-}G/{\widetilde{G}}$, and $dV^{\mathbb F}={\widetilde{\cal E}}_{-}{\widetilde{G}}^{-1}dD^{o,\mathbb F}$, we derive 
\begin{equation*} \begin{split}
d\widehat{\Gamma}&=Y^{\mathbb F}d{\widetilde{\cal E}}^{-1}+{1\over{{\widetilde{\cal E}}_{-}}}dY^{\mathbb F}={{\widehat{\Gamma}}\over{\widetilde G}}dD^{o,\mathbb F}+{1\over{{\widetilde{\cal E}}_{-}}}d{Y}^{\mathbb F}\\
&={{\widehat{\Gamma}}\over{\widetilde G}}dD^{o,\mathbb F}-{{f^{\mathbb{F}}}\over{{\widetilde{\cal E}}_{-}}}ds-{{h}\over{{\widetilde{\cal E}}_{-}}}dV^{\mathbb{F}}-{{1}\over{{\widetilde{\cal E}}_{-}}}dK^{\mathbb{F}}+{{Z^{\mathbb{F}}}\over{{\widetilde{\cal E}}_{-}}}dW\nonumber\\\
&={{\widehat{\Gamma}-h}\over{\widetilde G}}dD^{o,\mathbb F}-fds-{{1}\over{{\widetilde{\cal E}}_{-}}}dK^{\mathbb{F}}+{{Z^{\mathbb{F}}}\over{{\widetilde{\cal E}}_{-}}}dW.\label{equation4Gamma}
      \end{split}\end{equation*}
Thus, by  stopping $\widehat{\Gamma}$ and inserting the above equality in (\ref{YGGammainfinite}), we get 
\begin{equation}\label{SDE4YG}
d\widehat{Y} =-f(t)d(t\wedge\tau)-d\widehat{K}+d \widehat{M}+{\widehat{Z}}dW^{\tau},\quad\mbox{and}\quad \widehat{Y}_{\tau}=\xi.
\end{equation} 
This proves that $(\widehat{Y}, \widehat{Z}, \widehat{K},\widehat{M})$ satisfies the first equation in (\ref{RBSDEGinfinite}). Furthermore, it is clear that   $
{Y}^{\mathbb{F}}_{t}\geq S_{t}^{\mathbb{F}}$ implies  the second condition in (\ref{RBSDEGinfinite}). To prove the Skorokhod condition (the last condition in (\ref{RBSDEGinfinite})), we combine the Skorokhod condition for $({Y}^{\mathbb F}, {Z}^{\mathbb F}, {K}^{\mathbb F})$ with $\widehat{Y}_{-}\geq S_{-}$ on $\Rbrack0,\tau\Rbrack$, and derive
  \begin{equation*}
0\leq \int_{0}^{\tau}({\widehat{Y}}_{t-}-S_{t-})d\widehat{K}_t=\int_{0}^{\tau}({Y}^{\mathbb F}_{t-}-S^{\mathbb F}_{t-}){\widetilde{\cal E}}_{t-}^{-2}d{K}^{\mathbb F}_t\leq 0,\ P\mbox{-a.s..}
 \end{equation*}
This ends the second part, and the proof of the theorem is complete.
\end{proof}
\begin{proof}[Proof of Theorem \ref{EstimateInfinite} ]  We start by remarking that, due to It\^o's calculations, $\left(Y^{\mathbb{G}},Z^{\mathbb{G}},K^{\mathbb{G}},M^{\mathbb{G}}\right)$ is the unique solution to  (\ref{RBSDEGinfinite}) if and only if 
$$\left(\widetilde{Y}^{\mathbb{G}},\widetilde{Z}^{\mathbb{G}},\widetilde{K}^{\mathbb{G}},\widetilde{M}^{\mathbb{G}}\right):=\left(Y^{\mathbb{G}}\sqrt[p]{\widetilde{\cal E}},Z^{\mathbb{G}}\sqrt[p]{\widetilde{\cal E}_{-}},\sqrt[p]{\widetilde{\cal E}_{-}}\is{K}^{\mathbb{G}},\sqrt[p]{\widetilde{\cal E}_{-}}\is{M}^{\mathbb{G}}\right)$$ is the unique solution to (\ref{EquivalentRBSDE}). Furthermore, under (\ref{MainAssumption4InfiniteHorizon}), this solution is an $L^p(P,\mathbb G)$-solution as soon as assertion (b) holds. In fact, this is due to 
 \begin{eqnarray*}
 \Vert \widetilde{Y}^{\mathbb{G}}\Vert_{\mathbb{D}_{\tau}(P,p)}=\Vert{Y}^{{\mathbb{G}}}\Vert_{\widetilde{\mathbb{D}}_{\tau}(P,p)},\  \Vert \widetilde{Z}^{\mathbb{G}}\Vert_{\mathbb{S}_{\tau}(P,p)}=\Vert{Z}^{{\mathbb{G}}}\Vert_{\widetilde{\mathbb{S}}_{\tau}(P,p)}.
 \end{eqnarray*}
Thus, on the one hand, assertion (d) follows immediately as soon as assertions (a) and (b) hold. On the other hand, the uniqueness of the solution to  (\ref{RBSDEGinfinite}) is a direct consequence of assertion (c). Therefore, the rest of the proof focuses on proving existence of a solution to  (\ref{RBSDEGinfinite}) (i.e. the first half of assertion (a)), and assertions (b) and (c) in four parts.\\
{\bf Part 1.} Here, in this part, we consider an $\mathbb F$-stopping time $\sigma$, and we assume that there exists a positive constant $C\in (0,+\infty)$ such that 
\begin{eqnarray}\label{BoundednessAssumption}
\max\left(\vert{h}\vert,\int_0^{\cdot}\vert f_s\vert ds, \sup_{0\leq u\leq\cdot}S_u^+\right)\leq C {\cal E}(G_{-}^{-1}\is{ m})^{1/p},\quad\mbox{on}\quad\Lbrack0,\sigma\Rbrack.\end{eqnarray}
Our goal, in this part, lies in proving under this assumption that there exists a solution to  (\ref{RBSDEGinfinite}) and assertion (b) holds for $\Lbrack0,\sigma\wedge\tau\Rbrack$. To this end, we consider the sequence  of data $(f^{(n)}, h^{(n)}, S^{(n)})$ given for $n\geq 1$ by 
\begin{equation}\label{Sequence4[0,n]}
f^{(n)}:=fI_{\Lbrack0,n\wedge\sigma\Rbrack},\ h^{(n)}_t:=h_{t\wedge\sigma\wedge{n}},\ S^{(n)}_t:=S_{n\wedge{t}\wedge\sigma},\ \xi^{(n)}:=h_{n\wedge\sigma\wedge\tau}.
\end{equation}
For any $n\geq 1$, thanks to Theorem \ref{abcde}-(b) and Remark \ref{Extension4Section4toSigma}-(a), the RBSDE (\ref{RBSDEGinfinite}) associated to $(f^{(n)},S^{(n)},\xi^{(n)})$ has a unique solution denoted by $({Y}^{\mathbb{G},n},Z^{\mathbb{G},n},M^{\mathbb{G},n}, K^{\mathbb{G},n})$ for the horizon $ n\wedge\sigma\wedge\tau$ (i.e., the case when $T=n\wedge\sigma$). For any $n, m\geq 1$, we apply  Theorem \ref{estimates} to each $(f^{(n)}, h^{(n)}, S^{(n)})$ and apply Theorem \ref{estimates1} to the triplet $(\delta{f}, \delta{h},\delta{ S},\delta\xi)$ given by
\begin{equation*}
\begin{split}
\delta{f}:=f^{(n)}-f^{(n+m)},\quad \delta{h}:= h^{(n)}-h^{(n+m)},\quad\delta{ S}:=S^{(n)}-S^{(n+m)},\quad\delta\xi:=\xi^{(n)}-\xi^{(n+m)},\end{split}
\end{equation*}
using the bounded horizon $T=(n+m)\wedge\sigma$ for both theorems. This yields
\begin{equation}\label{Inequality4Limit}      \begin{split}
&E\left[\sup_{0\leq t\leq T}\widetilde{\cal E}_t\vert{Y}^{\mathbb{G},n}_{t}\vert^p+\left(\int_{0}^{T}(\widetilde{\cal E}_{s-})^{{{2}\over{p}}}\vert Z^{\mathbb{G},n}_{s}\vert^{2}ds\right)^{{{p}\over{2}}}\right]\\
&+E\left[\left((\widetilde{\cal E}_{-})^{1/p}\is{ K}^{\mathbb{G},n}_T\right)^p+((\widetilde{\cal E}_{-})^{2/p}\is[ M^{\mathbb{G},n}, M^{\mathbb{G},n}]_T)^{{{p}\over{2}}}\right]\\
&\leq \widetilde{C} E^{\widetilde{Q}}\left[\left(\int_{0}^{n\wedge\sigma\wedge\tau}\vert f(s)\vert ds\right)^p +\sup_{0\leq s \leq n\wedge\sigma\wedge\tau}(S^{+}_{s})^p+\vert\xi^{(n)}\vert^p\right],
      \end{split}\end{equation}
and 
 \begin{equation}\label{Inequa4Convergence}
 \begin{split}
&\Vert{Y}^{\mathbb{G},n}-Y^{\mathbb{G},n+m}\Vert_{\widetilde{\mathbb{D}}_{\tau}(P,p)}+\Vert{ Z}^{\mathbb{G},n}-Z^{\mathbb{G},n+m}\Vert_{\widetilde{\mathbb{S}}_{\tau}(P,p)}+\Vert\sqrt[p]{\widetilde{\cal E}_{-}}\is(M^{\mathbb{G},n}-M^{\mathbb{G},n+m})^{\tau}\Vert_{{\cal{M}}^p(P)} \\
&\leq  \widetilde{C}_1 \Delta_{\widetilde{Q}}(\xi^{(n)}-\xi^{(n+m)},{f}^{(n)}-f^{(n+m)},{S}^{(n)}-S^{(n+m)})\\
&+ \widetilde{C}_2\sqrt{\Vert\sup_{0\leq t\leq\tau\wedge\sigma}\vert{S}_{t\wedge{n}}-S_{t\wedge{(n+m)}}\vert\Vert_{L^p(\widetilde{Q})}\sum_{i\in\{n,n+m\}}\Delta_{\widetilde{Q}}(\xi^{(i)}, f^{(i)},(S^{(i)})^+)},\end{split}
\end{equation}
where $\Delta_{\widetilde{Q}}(\xi^{(i)}, f^{(i)},(S^{(i)})^+)$ is given via (\ref{Delta(xi,f,S)}), and which we recall below:
\begin{eqnarray*}
\Delta_{\widetilde{Q}}(\xi^{(i)}, f^{(i)},(S^{(i)})^+):=\Vert\xi^{(i)}\Vert_{L^p(\widetilde{Q})}+\Vert\int_{0}^{T\wedge\tau}\vert{f}^{(i)}(s)\vert ds\Vert_{L^p(\widetilde{Q})} +\Vert({S}^{(i)})^+\Vert_{\mathbb{S}_{T\wedge\tau}(\widetilde{Q})}.\end{eqnarray*}
Next, we calculate the limits, when $n$ and/or $m$ go to infinity, of the right-hand-sides of the inequalities (\ref{Inequality4Limit}) and (\ref{Inequa4Convergence}). To this end, we start by applying Lemma \ref{ExpecationQtilde2P} to $\left(\int_{0}^{\cdot}\vert f(s)\vert ds\right)^p$, $\sup_{0\leq s \leq \cdot}(S^{+}_{s})^p$, and $\vert{h}\vert^p$,  and get
\begin{equation}\label{Limits4RightHandSide}
\begin{split}
&\lim_{n\to\infty}E^{\widetilde{Q}}\left[\left(\int_{0}^{n\wedge\sigma\wedge\tau}\vert f(s)\vert ds\right)^p\right]=G_0E\left[\int_0^{\infty} (F_{t\wedge\sigma})^p dV^{\mathbb F}_t\right],\\
&\lim_{n\to\infty}E^{\widetilde{Q}}\left[\sup_{0\leq s \leq n\wedge\sigma\wedge\tau}(S^{+}_{s})^p\right]=G_0E\left[\int_0^{\infty} \sup_{0\leq s \leq{t\wedge\sigma}}(S^{+}_{s})^pdV^{\mathbb F}_t\right],\\
&\lim_{n\to\infty}E^{\widetilde{Q}}\left[\vert\xi^{(n)}\vert^p\right]=G_0E\left[\int_0^{\infty} \vert{h}_{t\wedge\sigma}\vert^p dV^{\mathbb F}_t\right].
\end{split}
\end{equation} 
This determines the limits for the right-hand-side terms of  (\ref{Inequality4Limit}). To address the limits of the right-hand-side terms of (\ref{Inequa4Convergence}), we remark that 
\begin{equation*}
\Vert\int_{0}^{T\wedge\tau}\vert{f}^{(n)}(s)-f^{(n+m)}(s)\vert {d}s\Vert_{L^p(\widetilde{Q})}^p=\Vert \int_0^{\tau\wedge(n+m)}I_{\Lbrack0,\sigma\Rbrack\cap\Rbrack{n},\infty\Lbrack}(s)\vert{f}(s)\vert{d}s\Vert_{L^p(\widetilde{Q})}^p.
\end{equation*}
increases with $m$. Thus, put $F^{(n)}(s):=(F(s\wedge\sigma)-F(s\wedge{n}\wedge\sigma))^p$ and by applying Lemma  \ref{ExpecationQtilde2P}, we obtain 
\begin{equation*}
\sup_m\Vert\int_{0}^{\tau}\vert{f}^{(n)}(s)-f^{(n+m)}(s)\vert {d}s\Vert_{L^p(\widetilde{Q})}^p=G_0E\left[\int_0^{\infty}F^{(n)}(s){d}V^{\mathbb{F}}_s\right].
\end{equation*}
By combining this equality with (\ref{MainAssumption4InfiniteHorizon}) and the dominated convergence theorem, we deduce that 
\begin{equation}\label{LimitZero1}
\displaystyle\lim_{n\to\infty}\sup_{m\geq1}\Vert\int_{0}^{\tau}\vert{f}^{(n)}(s)-f^{(n+m)}(s)\vert {d}s\Vert_{L^p(\widetilde{Q})}^p=0.\end{equation}
Similar arguments allow us to conclude that
\begin{equation*}
\Vert{S}^{(n)}-S^{(n+m)}\Vert_{\mathbb{D}_{T\wedge\tau}(\widetilde{Q},p)}^p=E^{\widetilde{Q}}\left[\sup_{n<t\leq(n+m)\wedge\sigma\wedge\tau}\vert{S}_{n}-S_t\vert^p I_{\{\tau\wedge\sigma>n\}}\right],\end{equation*}
increases with $m$ to $\displaystyle{G_0}E\left[\int_{0}^{\infty}\sup_{n<u\leq\sigma\wedge{t}}\vert{S}_{n}-S_u\vert^pI_{\{\sigma\wedge{t}>n\}}{d}V^{\mathbb F}_t\right]$, and hence 
\begin{equation}\label{LimitZero2}
\displaystyle\lim_{n\to\infty}\sup_{m\geq1}\Vert{S}^{(n)}-S^{(n+m)}\Vert_{\mathbb{D}_{T\wedge\tau}(\widetilde{Q},p)}^p=0.\end{equation}
Thanks to (\ref{BoundednessAssumption}), we derive
\begin{equation*}      \begin{split}
E^{\widetilde{Q}}\left[\vert\xi^{(n)}-\xi^{(n+m)}\vert^p\right]&=E\left[\widetilde{Z}_{(n+m)\wedge\tau}\vert{h}_n-h_{\tau\wedge\sigma\wedge(n+m)}\vert^p{I}_{\{\tau\wedge\sigma>n\}}\right]\\
&\leq 2^pCP(\tau\wedge\sigma>n),      \end{split}\end{equation*}
and as a result we get
\begin{equation}\label{LimitZero3}
\displaystyle\lim_{n\to\infty}\sup_{m\geq1}E^{\widetilde{Q}}\left[\vert\xi^{(n)}-\xi^{(n+m)}\vert^p\right]\leq \lim_{n\to\infty} 2CP(\tau\wedge\sigma>n)=0.
\end{equation} 
Thus, by combining (\ref{LimitZero1}), (\ref{LimitZero2}) and (\ref{LimitZero3}), we conclude that the right-hand-side term of  (\ref{Inequa4Convergence}) goes to zero when $n$ goes to infinity uniformly in $m$. This proves that $(Y^{\mathbb{G},n}, Z^{\mathbb{G},n}, K^{\mathbb G,n},M^{\mathbb{G},n})$ is a Cauchy sequence in norm, and hence it converges to $(Y^{\mathbb{G}}, Z^{\mathbb{G}}, K^{\mathbb G},M^{\mathbb{G}})$ in norm and almost surely for a subsequence. On the one hand, we conclude that $(Y^{\mathbb{G}}, Z^{\mathbb{G}}, K^{\mathbb G},M^{\mathbb{G}})$ satisfies the firts equation in (\ref{RBSDEGinfinite}) and $Y^{\mathbb{G}}\geq S$ on $\Lbrack0,\tau\Lbrack$. The proof of the Skorokhod condition (the last condition in (\ref{RBSDEGinfinite})), is similar to the one in the proof of Theorem \ref{Relationship4InfiniteBSDE} (see the end of  part 2). On the other hand, by taking the limit in (\ref{Inequality4Limit})  and using Fatou and (\ref{Limits4RightHandSide}), the proof of assertion (b) follows immediately. This ends the first part. \\
 {\bf Part 2.} In this part we prove that assertion (c) holds under the assumption (\ref{BoundednessAssumption}) over the interval $\Lbrack0,\sigma\wedge\tau\Rbrack$, where $\sigma$ is an $\mathbb F$-stopping time. To this end, we consider two triplets $(f^{(i)}, S^{(i)},h^{(i)})$, $i=1,2$, which satisfy the boundedness assumption (\ref{BoundednessAssumption}), and to which we associate two  sequences $(f^{(i, n)}, S^{(i,n)}, h^{(i,n)})$, $i=1,2$, as in (\ref{Sequence4[0,n]}). On the one hand, by virtue of part 1, we deduce that for each $i=1,2$ $(Y^{\mathbb{G},(i,n)}, Z^{\mathbb{G},(i,n)}, K^{\mathbb G,(i,n)},M^{\mathbb{G},(i,n)})$ converges to  $(Y^{\mathbb{G},(i)}, Z^{\mathbb{G},(i)}, K^{\mathbb G,(i)},M^{\mathbb{G},(i)})$ in norm and almost surely for a subsequence, and hence this quadruplet limit is solution to  (\ref{RBSDEGinfinite}) for the horizon $\sigma\wedge\tau$. On the other hand, for each $n\geq 1$, we apply Theorem \ref{estimates1} for  
\begin{eqnarray*}\begin{split}
&\delta f^{(n)}:=f^{(1, n)}- f^{(2, n)},\quad\delta S^{(n)}:=S^{(1,n)}-S^{(2,n)},\quad\delta\xi^{(n)}:= \xi^{(1,n)}-\xi^{(2,n)},\\
& \mbox{and}\\
&\delta{Y}^{\mathbb{G},(n)}:=Y^{\mathbb{G},(1,n)}-Y^{\mathbb{G},(2,n)}, \quad \delta{ Z}^{\mathbb{G},(n)}:=Z^{\mathbb{G},(1,n)}-Z^{\mathbb{G},(2,n)},\\
& \delta{K}^{\mathbb G,(n)}:=K^{\mathbb G,(1,n)}-K^{\mathbb G,(2,n)},\quad \delta{M}^{\mathbb{G},(n)}:=M^{\mathbb{G},(1,n)})-M^{\mathbb{G},(2,n)},\end{split}
\end{eqnarray*}
 and get 
 \begin{equation}\label{Convergence4Differences}
 \begin{split}
& \Vert\delta{Y}^{\mathbb{G},(n)}\Vert_{\widetilde{\mathbb{D}}_{T\wedge\tau}(P,p)}+ \Vert\delta{Z}^{\mathbb{G},(n)}\Vert_{\widetilde{\mathbb{S}}_{T\wedge\tau}(P,p)}+ \Vert(\widetilde{\cal E}_{-})^{1/p}\cdot\delta{M}^{\mathbb{G},(n)}\Vert_{L^p(P)}\\
&\leq \widetilde{C}_1  \Delta_{\widetilde{Q}}(\delta\xi^{(n)}, \delta f^{(n)},\delta S^{(n)})+ \widetilde{C}_2\sqrt{\Vert\delta S^{(n)}\Vert_{\mathbb{S}(\widetilde{Q},p)}\Sigma_n},\end{split}
\end{equation}
where 
\begin{equation}\label{Sigma(n)}
\Sigma_n:=\sum_{i=1}^2\Delta_{\widetilde{Q}}\left(\xi^{(i,n)},{f}^{(i,n)},({S}^{(i,n)})^+\right).\end{equation}
Similarly, as in the proof of (\ref{Limits4RightHandSide}), we use Lemma \ref{ExpecationQtilde2P}  and the boundedness assumption  (\ref{BoundednessAssumption}) that each  triplet $(f^{(i)}, S^{(i)},h^{(i)})$ satisfies, and get
\begin{eqnarray}\label{Limits4Differences}
\begin{split}
&\displaystyle\lim_{n\to\infty}E^{\widetilde{Q}}\left[\left(\int_{0}^{n\wedge\sigma\wedge\tau}\vert \delta{f}^{(n)}_s)\vert ds\right)^p\right]=G_0E\left[\int_0^{\infty} \vert\delta{F}_{t\wedge\sigma}\vert^p dV^{\mathbb F}_t\right],\\ 
&\displaystyle\lim_{n\to\infty}E^{\widetilde{Q}}\left[\vert\delta\xi^{(n)}\vert^p\right]=G_0E\left[\int_0^{\infty} \vert\delta{h}_{t\wedge\sigma}\vert^p dV^{\mathbb F}_t\right],\\
&\displaystyle\lim_{n\to\infty}E^{\widetilde{Q}}\left[\sup_{0\leq s \leq\sigma\wedge\tau}\vert\delta{S}_{s}^{(n)}\vert^p\right]=G_0E\left[\int_0^{\infty} \sup_{0\leq s \leq{t\wedge\sigma}}\vert\delta{S}_{s}\vert^pdV^{\mathbb F}_t\right].\end{split}
\end{eqnarray}
 Thus, by taking the limit in (\ref{Convergence4Differences}), using Fatou's lemma for its left-hand-side term, and using (\ref{Limits4Differences}) for its right-hand-side term, assertion (c) follows immediately. This ends the second part.\\
{\bf Part 3.} In this part, we drop the assumption (\ref{BoundednessAssumption}) and prove existence of solution to  (\ref{RBSDEGinfinite})  and assertion (b). Hence, we consider the following sequence of stopping times
\begin{eqnarray*}
T_n:=\inf\left\{t\geq 0\ :\quad  {{\vert{h}_t\vert^p+\vert{S}_t\vert^p+(\int_0^t\vert f(s)\vert ds)^p}\over{{\cal E}_t(G_{-}^{-1}\cdot m)}} >n\right\},\end{eqnarray*}
and the sequences
\begin{eqnarray}\label{Consutrction4DataSequence}
h^{(n)}:=hI_{\Lbrack0, T_n\Lbrack},\ f^{(n)}:=fI_{\Lbrack0, T_n\Rbrack},\ S^{(n)}:=S I_{\Lbrack0, T_n\Lbrack},\ \xi^{(n)}:=h_{\tau}I_{\{\tau<T_n\}}.
\end{eqnarray}
Thus, for any $n\geq 1$, it is clear that the triplet $(f^{(n)}, h^{(n)}, S^{(n)})$ satisfies (\ref{BoundednessAssumption}) on $\Lbrack0, T_n\Rbrack$.  Thus, thanks to the first and the second parts, we deduce the existence of a unique solution to (\ref{RBSDEGinfinite}), denoted by $(Y^{\mathbb{G},(n)}, Z^{\mathbb{G},(n)}, K^{\mathbb G,(n)},M^{\mathbb{G},(n)})$, associated to $(f^{(n)}, h^{(n)}, S^{(n)})$ with the horizon $T_n\wedge\tau$, which remains a solution for any horizon $T_k\wedge\tau$ with $k\geq n$. Furthermore, we put 
$$
\Gamma(k):= \Vert{ F}^{(k)}+\vert{h}^{(k)}\vert+\sup_{0\leq u\leq \cdot}(S^{(k)}_u)^+\Vert_{L^p(P\otimes{V}^{\mathbb F})},\quad k\geq 1,$$
 and derive 
\begin{equation}\label{Estimate4PTinfinityproof}
\Vert{Y}^{\mathbb{G},n}\Vert_{\widetilde{\mathbb{D}}_{\tau}(P,p)}+\Vert\sqrt[p]{\widetilde{\cal E}_{-}}\is({M}^{\mathbb{G},n})^{\tau}\Vert_{{\cal{M}}^p(P)}
+\Vert{Z}^{\mathbb{G},n}\Vert_{\widetilde{\mathbb{S}}_{\tau}(P,p)}+\Vert\sqrt[p]{{\widetilde{\cal E}_{-}}}\is{K}^{\mathbb{G},n})_{\tau}\Vert_{L^p(P)}
\leq C\Gamma(n),
\end{equation}
 due to assertion (b) and for any $n\geq 1$ and $m\geq 1$
 \begin{equation}\label{Estimate4P1Tinifiniteproof}
  \begin{split}
 &
 \Vert{Y}^{\mathbb{G},n}-Y^{\mathbb{G},n+m}\Vert_{\widetilde{\mathbb{D}}_{\tau}(P,p)}+\Vert{ Z}^{\mathbb{G},n}-Z^{\mathbb{G},n+m}\Vert_{\widetilde{\mathbb{D}}_{\tau}(P,p)}
 +
\Vert\sqrt[p]{\widetilde{\cal E}_{-}}\is({M}^{\mathbb{G},n}-M^{\mathbb{G},n+m})^{\tau}\Vert_{{\cal{M}}^p(P)}
\\
\leq& C_1 \Vert\vert {h}^{(n)}-{h}^{(n+m)} \vert+\vert{ F}^{(n)}-{ F}^{(n+m)}\vert+\sup_{0\leq{u}\leq\cdot}\vert{S}_u^{(n)}-{S}_u^{(n+m)}\vert\Vert_{L^p(P\otimes{V}^{\mathbb F})}\\
&+C_2\sqrt{\Vert\sup_{0\leq{u}\leq\cdot}\vert{S}_u^{(n)}-{S}_u^{(n+m)}\vert\Vert_{L^p(P\otimes{V}^{\mathbb F})}\sum_{i\in\{n,n+m\}}\Gamma(i)}.
\end{split}
\end{equation}
Thanks to part 2, this latter inequality follows from assertion (c) applied to 
\begin{equation*}
\begin{split}
&\delta{Y}^{\mathbb{G}}:=Y^{\mathbb{G},(n)}-Y^{\mathbb{G},(n+m)},\ \delta{Z}^{\mathbb{G}}:=Z^{\mathbb{G},(n)}-Z^{\mathbb{G},(n+m)},\  \delta{S}:=S^{(n)}-S^{(n+m)},\\
& \delta{K}^{\mathbb{G}}:=K^{\mathbb{G},(n)}-K^{\mathbb{G},(n+m)},\ \delta{M}^{\mathbb{G}}:=M^{\mathbb{G},(n)}-M^{\mathbb{G},(n+m)},\label{deltaProcesses}\\
&\delta{h}:=h^{(n)}-h^{(n+m)} ,\ \delta{F}:=\int_0^{\cdot}\vert{f}^{(n)}_s-{f}^{(n+m)}_s \vert{d}s,\quad F^{(i)}:=\int_0^{\cdot}\vert{f}^{(i)}_s\vert{d}s,\quad .\end{split}
\end{equation*}
Then by virtue of (\ref{MainAssumption4InfiniteHorizon}) and the dominated convergence theorem, we derive 
\begin{equation*}\begin{split}
&\lim_{n\to+\infty}\sup_{m\geq 1} \Vert\vert {h}^{(n)}-{h}^{(n+m)} \vert+\vert{ F}^{(n)}-{ F}^{(n+m)}\vert+\sup_{0\leq{u}\leq\cdot}\vert{S}_u^{(n)}-{S}_u^{(n+m)}\vert\Vert_{L^p(P\otimes{V}^{\mathbb F})}\\
&\leq\lim_{n\to+\infty}\Vert{I}_{\Lbrack{T_n},+\infty\Lbrack}(\vert {h}\vert+{ F}+\sup_{0\leq{u}\leq\cdot}\vert{S}_u\vert)\Vert_{L^p(P\otimes{V}^{\mathbb F})}=0.\end{split}
\end{equation*}
A combination of this with (\ref{Estimate4P1Tinifiniteproof}) proves that $(Y^{\mathbb{G},(n)}, Z^{\mathbb{G},(n)}, K^{\mathbb G,(n)},M^{\mathbb{G},(n)})$ is a Cauchy sequence in norm, and hence it converges  to $(Y^{\mathbb{G}}, Z^{\mathbb{G}}, K^{\mathbb G},M^{\mathbb{G}})$ in norm and almost surely for a subsequence. As a result, $(Y^{\mathbb{G}}, Z^{\mathbb{G}}, K^{\mathbb G},M^{\mathbb{G}})$ clearly satisfies the first equation in (\ref{RBSDEGinfinite}) and $Y^{\mathbb{G}}\geq S$ on $\Lbrack0,\tau\Lbrack$, while the last condition in  (\ref{RBSDEGinfinite})  can be proved as in part 2. Hence, due to Fatou's lemma and (\ref{Estimate4PTinfinityproof}), we conclude that assertion (b) holds. This ends part 3.\\
{\bf Part 4.} Here we prove assertion (c) under no assumption. Thus, we consider a pair of data $(f^{(i)}, S^{(i)},h^{(i)}, \xi^{(i)})$, $i=1,2$, to which we associate two sequences of $\mathbb F$-stopping times $(T^{(i)}_n)_n$ for $i=1,2$ as in part 3, and two data-sequences $(f^{(n,i)}, h^{(n,i)}, S^{(n,i)})$ which are constructed from $(f^{(i)}, h^{(i)}, S^{(i)})$ and $T_n:=\min(T^{(1)}_n, T^{(2)}_n)$ via (\ref{Consutrction4DataSequence}). On the one hand, thanks to part 2, we obtain the existence of  $(Y^{\mathbb{G},(n,i)}, Z^{\mathbb{G},(n,i)}, K^{\mathbb G,(n,i)},M^{\mathbb{G},(n,i)})_{n\geq 1}$  ($i=1,2$) solution to  (\ref{RBSDEGinfinite}) for the data $(f^{(n,i)}, h^{(n,i)}, S^{(n,i)})$ with the horizon $T_n\wedge\tau$. Furthermore, the sequence $$(Y^{\mathbb{G},(n,i)}, Z^{\mathbb{G},(n,i)}, K^{\mathbb G,(n,i)},M^{\mathbb{G},(n,i)})_{n\geq1}$$ converges (in norm and almost surely for a subsequence) to  $(Y^{\mathbb{G},i}, Z^{\mathbb{G},i}, K^{\mathbb G,i},M^{\mathbb{G},i})$, which is a solution to  (\ref{RBSDEGinfinite})  for  $(f^{(i)}, S^{(i)},h^{(i)}, \xi^{(i)})$ and the horizon $\tau$. On the other hand, thanks to part 3, we apply assertion (c) to  the quadruplet $(\delta{Y}^{\mathbb{G},(n)}, \delta{Z}^{\mathbb{G},(n)}, \delta{K}^{\mathbb G,(n)},\delta{M}^{\mathbb{G},(n)})$ given by 
\begin{equation*}
\begin{split}
&\delta{Y}^{\mathbb{G},(n)}:=Y^{\mathbb{G},(n,1)}-Y^{\mathbb{G},(n,2)},\quad \delta{Z}^{\mathbb{G},(n)}:= Z^{\mathbb{G},(n,1)}-Z^{\mathbb{G},(n,2)},\\
& \delta{K}^{\mathbb G,(n)}:= K^{\mathbb G,(n,1)}-K^{\mathbb G,(n,2)},\quad \delta{M}^{\mathbb{G},(n)}):=M^{\mathbb{G},(n,1)})-M^{\mathbb{G},(n,2)},\end{split}\end{equation*}
which is associated to  
\begin{equation*}
\begin{split}
(\delta{f}^{(n)}, \delta{h}^{(n)}, \delta{S}^{(n)})&:=(f^{(n,1)}-f^{(n,2)}, h^{(n,1)}-h^{(n,2)}, S^{(n,1)}-S^{(n,2)})\\
&= (\delta{f}, \delta{h}, \delta{S})I_{\Lbrack0,T_n\Lbrack}.\end{split}\end{equation*}
This yields
 \begin{align*}
 &
 \Vert\delta Y^{{\mathbb{G},(n)}}\Vert_{\widetilde{\mathbb{D}}_{\tau}(P,p)}+\Vert \delta Z^{\mathbb{G},(n)}\Vert_{\widetilde{\mathbb{S}}_{\tau}(P,p)}
 +\Vert\sqrt[p]{\widetilde{\cal E}_{-}}\is\delta M^{\mathbb{G},(n)}\Vert_{{\cal{M}}^p(P)}\\
&\leq C_1 \Vert \vert\delta{h}^{(n)}\vert+\vert{\delta{F}^{(n)}}\vert+\sup_{0\leq{u}\leq{\cdot}}\vert\delta{S}^{(n)}_u\vert\Vert_{L^p(P\otimes V^{\mathbb F})}\\
&\quad\quad+C_2\sqrt{\Vert\sup_{0\leq{u}\leq{\cdot}}\vert\delta{S}^{(n)}_u\vert\Vert_{L^p(P\otimes V^{\mathbb F})}} \sqrt{\sum_{i=1}^2\Gamma(i,n)},
\end{align*} 
where 
$$\Gamma(i,n):=\Vert {F}^{(n,i)}+\vert{h}^{(n,i)}\vert+\sup_{0\leq{u}\leq{\cdot}}({S}_u^{(n,i)})^+\Vert_{L^p(P\otimes V^{\mathbb F})},\quad i=1, 2.$$
Then by  taking the limits in both sides of this inequality, and using Fatou on the left-hand-side term and the convergence monotone theorem on the right-hand-side term afterwards, we conclude that assertion (c) follows immediately for this general case.  This ends the fourth part, and the proof of the theorem is complete.
\end{proof}
\begin{appendices}

\section{A martingale inequality }
The following lemma, that plays a crucial role in our estimations, is interesting in itself and  generalizes \cite[Lemma 4.8]{Choulli4} . 
\begin{lemma}\label{Lemma4.8FromChoulliThesis}Let $r^{-1}=a^{-1}+b^{-1},$ where $ a>1$ and $b>1$. Then there exists a positive constant $\kappa=\kappa(a,b)$ depending only on $a$ and $b$  such that for any triplet $(H, X, M)$, where $H$ is predictable, $X$ is RCLL and adapted process, $M$ is a martingale,  and $\vert{H}\vert \leq \vert{X_{-}}\vert$, the following inequality holds. 
\begin{eqnarray*}
\Vert \sup_{0\leq{t}\leq{T}}\vert(H\is{M})_t\vert\Vert_r\leq \kappa\Vert\sup_{0\leq{t}\leq{T}}\vert{X}_t\vert\Vert_a\Vert[M]_T^{\frac{1}{2}}\Vert_b.
\end{eqnarray*}
\end{lemma}
\begin{proof} When $H=X_{-}$, the assertion can be found in \cite[Lemma 4.8]{Choulli4}. To prove the general case, we remark that, there is no loss of generality in assuming $\vert X_{-}\vert>0$, and hence $H/X_{-}$ is a well defined process, which is predictable and bounded. Thus, put $\overline{M}:=(H/X_{-})\is{ M},$
and remark that $[\overline{M},\overline{M}]=(H/X_{-})^2\is[M, M]\leq [M, M]$. As a result, we get 
\begin{align*}
\Vert\sup_{0\leq{t}\leq{T}}\vert({H}\is{ M})_t\vert\Vert_r&=\Vert\sup_{0\leq{t}\leq{T}}\vert({X_{-}}\is\overline{M})_t\vert\Vert_r\leq \kappa \Vert\sup_{0\leq{t}\leq{T}}\vert{X}_t\vert\Vert_a\Vert[\overline{M},\overline{M}]_T^{\frac{1}{2}}\Vert_b\\
&\leq \kappa\Vert\sup_{0\leq{t}\leq{T}}\vert{X}_t\vert\Vert_a\Vert[M, M]_T^{\frac{1}{2}}\Vert_b.\end{align*}
This ends the proof of the lemma.
\end{proof}
\section{Proof of Lemmas  \ref{Lemma4.11}, \ref{ExpecationQtilde2Pbis}, \ref{technicallemma1}, and  \ref{ExpecationQtilde2P}}\label{Appendix4Proofs}
\begin{proof}[Proof of Lemma \ref{Lemma4.11}] 
Thanks to Lemma \ref{G-projection},  on $(\tau>s)$ we derive 
 \begin{equation*}
 \begin{split}
    & E^{\widetilde{Q}}\left[D^{o,\mathbb{F}}_{T\wedge\tau} -D^{o,\mathbb{F}}_{s-}\big|{\cal G}_{s}\right]=\Delta D^{o,\mathbb{F}}_{s}+ E\left[\int_{s\wedge\tau}^{T\wedge\tau} {\widetilde{Z}}_u{d}D^{o,\mathbb{F}}_u \big|{\cal G}_{s}\right]{\widetilde{Z}}_{s\wedge\tau}^{-1}\\
     &=E\left[\int_{s\wedge\tau}^{T\wedge\tau} \widetilde{Z}_u{d} D^{o,\mathbb{F}}_u \big|{\cal F}_{s}\right](\widetilde{Z}_{s\wedge\tau}G_s)^{-1}+\Delta D^{o,\mathbb{F}}_{s}\\
     &=E\left[\int_{s}^{T}{{ {\cal E}_{u-}(-{\widetilde G}^{-1}_{-}\is{ D}^{o,\mathbb F})}\over{{\cal E}_{s}(-{\widetilde G}^{-1}_{-}\is{ D}^{o,\mathbb F})}} d D^{o,\mathbb{F}}_u \big|{\cal F}_{s}\right]+\Delta D^{o,\mathbb{F}}_{s}\leq {\widetilde{G}}_{s}.\end{split}
     \end{equation*}
     This proves assertion (a). The rest of this proof addresses assertion (b). By combining (\ref{Vepsilon}) and $1-(1-x)^a\leq\max(a,1) x$ for any $0\leq{x}\leq 1$, we get 
     \begin{eqnarray*}
    \Delta\widetilde{V}^{(a)}= 1-\left(1-{{\Delta D^{o,\mathbb F}}\over{\widetilde G}}\right)^a\leq \max(1,a){{\Delta D^{o,\mathbb F}}\over{\widetilde G}}.
     \end{eqnarray*}
     Hence, by putting $ W:= \max(1,a){\widetilde G}^{-1}\is{ D}^{o,\mathbb F}- \widetilde{V}^{(a)},$
     we deduce that both 
     \begin{eqnarray*}
I_{\{\Delta D^{o,\mathbb F}\not=0\}} \is{ W}=\sum\left\{ \max(1,a){{\Delta D^{o,\mathbb F}}\over{\widetilde G}}- 1+\left(1-{{\Delta D^{o,\mathbb F}}\over{\widetilde G}}\right)^a\right\}
 \end{eqnarray*}
 and $ I_{\{\Delta D^{o,\mathbb F}=0\}} \is{ W}= (1-a)^+{\widetilde G}^{-1}I_{\{\Delta D^{o,\mathbb F}=0\}}\is{D}^{o,\mathbb F}$ are nondecreasing processes. Hence assertion (b)  follows immediately from this latter fact and $W= I_{\{\Delta D^{o,\mathbb F}=0\}} \is{ W}+ I_{\{\Delta D^{o,\mathbb F}\not=0\}} \is{ W}$. This ends the proof of the lemma.  \end{proof}
 \begin{proof}[Proof of Lemma \ref{ExpecationQtilde2Pbis}] Remark that, for any process $H$, we have 
     $$H_{T\wedge\tau}=H_{\tau}I_{\{0<\tau\leq T\}}+H_T I_{\{\tau>T\}}+H_0I_{\{\tau=0\}}.$$
     Thus, by applying this to $X/{\cal E}(G_{-}^{-1}\is{ m})=X/\widetilde{Z}$, we derive
     \begin{equation*}\begin{split}
     E^{\widetilde Q}[X_{T\wedge\tau}]&=E\left[{{X_{T\wedge\tau}}\over{\widetilde{Z}_{T\wedge\tau}}}\right]=E\left[{{X_{\tau}}\over{\widetilde{Z}_{\tau}}}I_{\{0<\tau\leq T\}}+{{X_T}\over{\widetilde{Z}_T}}I_{\{\tau> T\}} +X_0I_{\{\tau=0\}}\right]\\
     &=E\left[\int_0^T {X_s}\widetilde{Z}_s^{-1}dD_s^{o,\mathbb F}+X_T{\widetilde{Z}}_T^{-1}G_T+X_0(1-G_0)\right]\\
      &= E\left[G_0\int_0^T X_sdV_s^{\mathbb F}+G_0X_T{\widetilde {\cal E}}_T+X_0(1-G_0)\right].\end{split}
     \end{equation*}
     This proves assertion (a). To prove assertion (b), we recall that $dV_s^{\mathbb F}=-d{\widetilde {\cal E}}_s$ and we apply It\^o to ${\widetilde {\cal E}}X$ to get 
     $$\int_0^T X_sdV_s^{\mathbb F}+X_T{\widetilde {\cal E}}_T-X_0=\int_0^T {\widetilde {\cal E}}_{s-} dX_s.$$
      Thus, by  inserting this in (\ref{XunderQtilde}) assertion (b) follows  immediately.\\ 
      To prove assertion (c), we consider a nondecreasing RCLL and $\mathbb{F}$-adapted process $X$ such that $X_0=0$. Then using $dV^{\mathbb{F}}=-d\widetilde{\cal{E}}$, and $(\vert{h}\vert\is{V}^{\mathbb{F}})^r\leq \vert{h}\vert^r\is{V}^{\mathbb{F}}$ due to Jensen's inequality and $V^{\mathbb{F}}\leq 1$, we derive 
      \begin{equation*}
      \begin{split}
      \left(\widetilde{\cal{E}}_{-}\is X\right)^r&=   \left(\widetilde{\cal{E}}X- X\is\widetilde{\cal{E}}\right)^r\leq 2^r \left(\widetilde{\cal{E}}^rX^r+(X\is{V}^{\mathbb{F}})^r\right)\leq 2^r \left(\widetilde{\cal{E}}X^r+X^r\is{V}^{\mathbb{F}}\right)=2^r  \left(\widetilde{\cal{E}}_{-}\is{X}^r\right).
      \end{split}
      \end{equation*}
      Thus, by combining this inequality and assertion (b) (i.e. the equality (\ref{XunderQtilde2})) applied to the nondecreasing process with null initial value $\widetilde{\cal{E}}_{-}\is{X}^r$, the inequality (\ref{XunderQtilde3}) follows immediately. This proves assertion (c), and the proof of the lemma is complete. \end{proof}
     \begin{proof}[Proof of Lemma \ref{technicallemma1}]This proof has four parts where we prove the four assertions respectively. \\
    {\bf Part 1.} Let $a\in (0,+\infty)$ and $Y$ be a RCLL $\mathbb G$-semimartingale, and put $Y^*_t:=\sup_{0\leq s\leq t}\vert Y_s\vert$. Then, on the one hand,  we remark that 
\begin{eqnarray}\label{remark1}
\sup_{0\leq t\leq{T}\wedge\tau} {\widetilde{\cal E}}_t\vert Y_t\vert^a\leq \sup_{0\leq t\leq{T}\wedge\tau} {\widetilde{\cal E}}_t (Y_t^*)^a.
\end{eqnarray}
On the other hand, thanks to It\^o, we derive 
\begin{eqnarray}\label{remark2}
{\widetilde{\cal E}}(Y^*)^a=(Y_0^*)^a+{\widetilde{\cal E}}\is(Y^*)^a+(Y^*_{-})^a\is {\widetilde{\cal E}}\leq (Y_0^*)^a+{\widetilde{\cal E}}\is (Y^*)^a.
\end{eqnarray}
By combining (\ref{remark1}) and (\ref{remark2}) with ${\widetilde{\cal E}}=G/\left(G_0{\cal E}(G_{-}^{-1}\is{ m})\right)$, we get 
\begin{equation*}
\begin{split}
E\left[\sup_{0\leq t\leq{T}\wedge\tau} {\widetilde{\cal E}}_t\vert Y_t\vert^a\right]&\leq{E}\left[(Y_0^*)^a+\int_0^{T\wedge\tau} {\widetilde{\cal E}_s} d(Y^*_s)^a\right]\nonumber \\
&=E^{\widetilde{Q}}\left[(Y_0^*)^a+\int_0^{T\wedge\tau} {{G_s}\over{G_0}}{d}(Y^*_s)^a\right]\leq {1\over{G_0}}E^{\widetilde{Q}}\left[(Y^*_{T\wedge\tau})^a\right].\end{split}
\end{equation*}
This proves assertion (a). \\
{\bf Part 2.} Let $a\in (0,+\infty)$ and  $K$ be a RCLL nondecreasing and $\mathbb G$-optional process with $K_0=0$. Then, we remark that 
\begin{equation}\label{equa300}
\widetilde{\cal E}_{-}^a \is{ K}=K\widetilde{\cal E}^a-K\is\widetilde{\cal E}^a=K\widetilde{\cal E}^a+K{\widetilde{\cal E}_{-}^a}\is\widetilde{V}^{(a)}=K\widetilde{\cal E}^a+K_{-}{\widetilde{\cal E}_{-}^a}\is\widetilde{V}^{(a)}+\Delta{K}{\widetilde{\cal E}_{-}^a}\is\widetilde{V}^{(a)},\end{equation}
where $\widetilde {V}^{(a)}$ is defined in (\ref{Vepsilon}). As a result, by combining the above equality,  the fact that $(\sum_{i=1}^n x_i)^{1/a}\leq n^{1/a}\sum _{i=1}^n x_i^{1/a}$ for any sequence of nonnegative numbers and Lemma \ref{Lemma4.11}, we derive  
\begin{equation*}\begin{split}
E\left[(\widetilde{\cal E}_{-}^a \is{ K}_{T\wedge\tau})^{1/a}\right]&\leq 3^{1/a}E\left[(K_{T\wedge\tau})^{1/a}\widetilde{\cal E}_{T\wedge\tau}+(K_{-}{\widetilde{\cal E}_{-}^a}\is\widetilde{V}^{(a)}_{T\wedge\tau})^{1/a}+({\widetilde{\cal E}_{-}^a}\Delta{K}\is\widetilde{V}^{(a)}_{T\wedge\tau})^{1/a}\right]\\
&\leq \sqrt[a]{3}E^{\widetilde{Q}}\left[\sqrt[a]{K_{T\wedge\tau}}{{G_{T\wedge\tau}}\over{G_0}}\right]+\sqrt[a]{3}E\left[4\sup_{0\leq{t}\leq{T\wedge\tau}}K_{t}^{1/a}{\widetilde{\cal E}_t}+\sqrt[a]{{\widetilde{\cal E}_{-}^a}\Delta{K}\is\widetilde{V}^{(a)}_{T\wedge\tau}}\right].\end{split}\end{equation*}
Then, due to  $K^{1/a}\widetilde{\cal E}\leq \widetilde{\cal E}\is{K}^{1/a}$ and ${\widetilde{\cal E}}=G/\left(G_0{\cal E}(G_{-}^{-1}\is{ m})\right)$, the above inequality leads to 
\begin{equation}\label{equa299}
\begin{split}
E\left[(\widetilde{\cal E}_{-}^a \is{ K}_{T\wedge\tau})^{1/a}\right]&\leq  {{\sqrt[a]{3}}\over{G_0}}E^{\widetilde{Q}}\left[\sqrt[a]{K_{T\wedge\tau}}\right]+4\sqrt[a]{3}E\left[\int_0^{T\wedge\tau}{\widetilde{\cal E}_t} d\sqrt[a]{K_{t}}\right]+\sqrt[a]{3}E\left[\sqrt[a]{{\widetilde{\cal E}_{-}^a}\Delta{K}\is\widetilde{V}^{(a)}_{T\wedge\tau}}\right]\\
&\leq  5{{3^{1/a}}\over{G_0}}E^{\widetilde{Q}}\left[(K_{T\wedge\tau})^{1/a}\right]+3^{1/a}E\left[({\widetilde{\cal E}_{-}^a}\Delta{K}\is\widetilde{V}^{(a)}_{T\wedge\tau})^{1/a}\right].\end{split}\end{equation}
Thus, it remains to deal with the last term in the right-hand-side term of the above inequality. To this end, we distinguish whether $a\geq 1$ or $a<1$. \\
The case when $a\geq 1$, or equivalently $1/a\leq 1$. Then we use the fact that $(\sum x_i)^{1/a}\leq \sum x_i^{1/a}$ for any sequence of nonnegative numbers, and get 
\begin{equation*}
\begin{split}
&E\left[(\Delta{K}{\widetilde{\cal E}_{-}^a}\is\widetilde{V}^{(a)}_{T\wedge\tau})^{1/a}\right]\\
&= E\left[\left(\sum_{0\leq t\leq _{T\wedge\tau}} \Delta{K}_t\widetilde{\cal E}_{t-}^a\Delta\widetilde{V}^{(a)}_t\right)^{1/a}\right]\leq E\left[\sum_{0\leq t\leq _{T\wedge\tau}} (\Delta{K}_t)^{1/a}\widetilde{\cal E}_{t-}(\Delta\widetilde{V}^{(a)}_t)^{1/a}\right]\\
&\leq a^{1/a} E\left[\sum_{0\leq t\leq _{T\wedge\tau}} (\Delta{K}_t)^{1/a}\widetilde{\cal E}_{t-}\right]= a^{1/a} E\left[\sum_{0\leq t\leq _{T\wedge\tau}} (\Delta{K}_t)^{1/a}{{\widetilde{G}_t}\over{G_t}}\widetilde{\cal E}_{t}\right] \\
&={{a^{1/a} }\over{G_0}}{E}^{\widetilde{Q}}\left[\sum_{0\leq t\leq _{T\wedge\tau}}\widetilde{G}_t(\Delta{K}_t)^{1/a}\right] .\end{split}
\end{equation*}
The last equality follows from $\widetilde{\cal E}/G=G_0^{-1}/{\cal E}(G_{-}^{-1}\is{ m})$. Thus, by combining this latter inequality with (\ref{equa299}), assertion (b) follows immediately for this case of $a\geq 1$. For the case of $a\in (0,1)$, or equivalently $1/a>1$, we use Lemma \ref{Lemma4.11} and derive 
\begin{align*}
&E\left[(\Delta{K}{\widetilde{\cal E}_{-}^a}\is\widetilde{V}^{(a)})_{T\wedge\tau}-(\Delta{K}{\widetilde{\cal E}_{-}^a}\is\widetilde{V}^{(a)})_{{t\wedge\tau}-}\bigg|\ \mathcal{G}_{t}\right]\\
&=E\left[\int_{t\wedge\tau}^{T\wedge\tau}\Delta{K_s}{\widetilde{\cal E}_{s-}^a}d\widetilde{V}^{(a)}_s + (\Delta{K_{t\wedge\tau}}{\widetilde{\cal E}_{{t\wedge\tau}-}^a}\Delta \widetilde{V}^{(a)})_{t\wedge\tau}\bigg|\ \mathcal{G}_{t}\right]\\
&\leq E\left[\int_{t\wedge\tau}^{T\wedge\tau}\sup_{0\leq u\leq s}\Delta{K_u}{\widetilde{\cal E}_{u-}^a}d\widetilde{V}^{(a)}_s + \sup_{0\leq u\leq{t\wedge\tau}} \Delta{K_u}{\widetilde{\cal E}_{u-}^a}\bigg|\ \mathcal{G}_{t}\right]\\
&=E\left[\int_{t\wedge\tau}^{T\wedge\tau}E[\widetilde{V}^{(a)}_{T\wedge\tau}-\widetilde{V}^{(a)}_{s-}\big|{\cal G}_{s}]d\sup_{0\leq u\leq s}\Delta{K_u}{\widetilde{\cal E}_{u-}^a} + \sup_{0\leq u\leq{t\wedge\tau}} \Delta{K_u}{\widetilde{\cal E}_{u-}^a}\bigg|\ \mathcal{G}_{t}\right]\\
&\leq E\left[ \sup_{0\leq u\leq{T\wedge\tau}} \Delta{K_u}{\widetilde{\cal E}_{u-}^a}\bigg|\ \mathcal{G}_{t}\right].\end{align*}
Therefore, a direct application of  \cite[Th\'eor\`eme 99, Chapter VI]{DellacherieMeyer80}, we obtain
\begin{align*}
E\left[\sqrt[a]{a\Delta{K}{\widetilde{\cal E}_{-}^a}\is\widetilde{V}^{(a)}_{T\wedge\tau}}\right]&\leq{E}\left[ \sup_{0\leq u\leq{T\wedge\tau}} {\widetilde{\cal E}_{u-}}\sqrt[a]{\Delta{K_u}}\right]\leq {E}\left[ \sum_{0\leq u\leq{T\wedge\tau}}  {\widetilde{\cal E}_{u-}}\sqrt[a]{\Delta{K_u}}\right]\\
&=G_0^{-1} E^{\widetilde{Q}}\left[ \sum_{0\leq u\leq{T\wedge\tau}} \widetilde{G}_u\sqrt[a]{\Delta{K_u}}\right].
\end{align*}
Hence, by combining this inequality with (\ref{equa299}), assertion (b) follows immediately in this case of $a\in (0,1)$, and the proof of assertion (b) is complete.\\
{\bf Part 3.}  Here we prove assertion (c). To this end, we consider $p>1$,  a $\mathbb G$-optional process $H$, and we apply assertion (b) to $K=H\is[N^{\mathbb G}, N^{\mathbb G}]$ with $a=2/p$, and get 
\begin{equation*}
   \begin{split}
    E\left[({\widetilde{\cal E}}_{-}^{2/p}H\is[N^{\mathbb G},N^{\mathbb G}])_{T\wedge\tau} ^{{{2}\over{p}}}\right]\leq {{\kappa(a)}\over{G_0}}  E^{\widetilde{Q}}\left[(H\is[N^{\mathbb G},N^{\mathbb G}]_{T\wedge\tau})^{{{2}\over{p}}}+ \sum_{0\leq t\leq {T\wedge\tau}}{\widetilde{G}_t}H^{{{2}\over{p}}}_t\vert\Delta{N}^{\mathbb G}\vert^p\right].
    \end{split} \end{equation*}
     Therefore, assertion (c) follows from combining this with $\vert\Delta{N}^{\mathbb G}\vert^{p-1}\leq 1$ and  $$\sum_{0\leq{t}\leq\cdot}{\widetilde{G}_t}H^{p/2}_t\vert\Delta{N}^{\mathbb G}_t\vert= {\widetilde{G}}H^{p/2}\is\rm{Var}(N^{\mathbb G}).$$
 {\bf Part 4.}  Consider $p>1$ and a nonnegative and $\mathbb H$-optional process $H$. Thus, by applying assertion (c), we obtain (\ref{Equality4MG}). Hence, to get (\ref{Equality4MGOptionalF}), we remark that $\rm{Var}(N^{\mathbb G})=(G/\widetilde{G})\is{D}+ {\widetilde{G}}^{-1}I_{\Rbrack0,\tau\Lbrack}\is{ D}^{o,\mathbb F}$, and due to the $\mathbb F$-optionality of $H$ we have 
    \begin{equation*} 
 E^{\widetilde{Q}}\left[{\widetilde{G}}\sqrt{H^p}\is\rm{Var}(N^{\mathbb G})_T\right]=2E\left[\int_0^T {{\sqrt{H^p_t}}\over{{\cal E}_t(G_{-}^{-1}\is{m})}}I_{\Rbrack0,\tau\Lbrack}(t)d D^{o,\mathbb F}_t\right]=2E^{\widetilde{Q}}\left[(\sqrt{H^p}I_{\Rbrack0,\tau\Lbrack}\is{D}^{o,\mathbb F})_T\right].
      \end{equation*}
    Therefore, by combining this with (\ref{Equality4MG}), assertion (d) follows immediately. This ends the proof of the lemma.\end{proof} 
 \begin{proof}[Proof of Lemma \ref{ExpecationQtilde2P}] Thanks to Lemma \ref{ExpecationQtilde2Pbis}-(a) and $X_0=0$, we have 
     $$H_{T\wedge\tau}=H_{\tau}I_{\{0<\tau\leq T\}}+H_T I_{\{\tau>T\}}+H_0I_{\{\tau=0\}}.$$
     Thus, by applying this to the process $X/{\cal E}(G_{-}^{-1}\is{ m})$, we derive
     \begin{equation*}\begin{split}
     E^{\widetilde Q}[X_{T\wedge\tau}]= E\left[G_0\int_0^T X_sdV_s^{\mathbb F}+G_0X_T{\widetilde {\cal E}}_T\right].\end{split}
     \end{equation*}
     Then by letting $T$ to go to infinity, the term $\int_0^T X_sdV_s^{\mathbb F}$ increases to $\int_0^{\infty} X_sdV_s^{\mathbb F}$, while the term $G_0X_T{\widetilde {\cal E}}_T=G_TX_T/{\cal{E}}_T(G_{-}^{-1}\is m)$ is bounded by $C$ and goes to zero in virtue of $G_{\infty-}=\lim_{t\longrightarrow+\infty}G_t=0$ $P$-a.s. ($\tau<+\infty$ $P$-a.s.). Thus, the lemma follows immediately from these two remarks and the dominated convergence and the convergence monotone theorems. \end{proof}
    \end{appendices}
\section*{Acknowledgements}
\noindent Tahir Choulli's research is fully supported by the Natural Sciences and Engineering Research Council of Canada (NSERC), Grant NSERC RGPIN-2019-04779. \\
Both authors contributed equally to this research.   




\end{document}